\documentclass[12pt,reqno]{amsart}
\usepackage{}
\usepackage{a4wide}
\allowdisplaybreaks \numberwithin{equation}{section}
\usepackage{color}

\usepackage[titletoc,title]{appendix}
\numberwithin{equation}{section}

\newtheorem{theorem}{Theorem}[section]
\newtheorem{proposition}[theorem]{Proposition}

\newtheorem{lemma}[theorem]{Lemma}

\theoremstyle{definition}

\theoremstyle{remark}

\newcommand{\pa}{\partial}

\newcommand{\la}{\lambda}

\newcommand{\ds}{\displaystyle}

\newcommand{\beam}{\begin{eqnarray}}
\newcommand{\eeam}{\end{eqnarray}}
\newcommand{\beao}{\begin{eqnarray*}}
\newcommand{\eeao}{\end{eqnarray*}}
\newcommand{\barr}{\begin{array}}
\newcommand{\earr}{\end{array}}
\newcommand{\beqq}{\begin{equation}}
\newcommand{\eeqq}{\end{equation}}

\begin{document}

\title[Solutions for fractional operator problem
 ]
{Solutions for fractional operator problem via local Pohozaev identities}

\author{Yuxia Guo,   Ting Liu and  Jianjun Nie
}

\address{Department of Mathematical Science, Tsinghua University, Beijing, P.R.China}
\email{yguo@tsinghua.edu.cn}

\address{Department of Mathematical Science, Tsinghua University, Beijing, P.R.China}
\email{liuting17@mails.tsinghua.edu.cn}

\address{Institute of Mathematics, Academy of Mathematics and Systems Science, Chinese Academy of Sciences, Beijing 100190, P.R.China}
\email{niejjun@126.com}

\thanks{This work is supported by NSFC(11771235,11801545)}
\maketitle

\begin{abstract}
We consider the following fractional Schr\"{o}dinger equation involving critical exponent:
\begin{equation*}
\left\{\begin{array}{ll}
(-\Delta)^s u+V(|y'|,y'')u=u^{2^*_s-1} \ \hbox{ in } \ \mathbb{R}^N, \\
u>0, \ y \in \mathbb{R}^N,
\end{array}\right. \eqno{(P)}
\end{equation*}
where $s\in(\frac{1}{2}, 1)$, $(y',y'')\in \mathbb{R}^2\times \mathbb{R}^{N-2}$, $V(|y'|,y'')$ is a bounded nonnegative function with a weaker symmetry condition. We prove the existence of infinitely many solutions for the above problem by  a finite dimensional reduction method combining various Pohazaev identies.

\end{abstract}

\vspace{0.5cm}


{\bf AMS} Subject Classification: 35B05; 35B45.

\section{Introduction}

In this paper, we are concerned with the following problem:
\begin{equation}\label{problem1}
\left\{\begin{array}{ll}
(-\Delta)^s u+V(|y'|,y'')u=u^{2^*_s-1}, \ \hbox{ in } \ \mathbb{R}^N, \\
u>0, \ y \in \mathbb{R}^N,
\end{array}\right.
\end{equation}
where $s\in(\frac{1}{2}, 1)$ and $2^*_s=\frac{2N}{N-2s}$ is the  critical Sobolev exponent. For any $s\in (0,1)$, $(-\Delta)^s$ is the fractional Laplacian in $\mathbb{R}^N$, which is a nonlocal operator defined as:
\begin{equation} \label{aum23}
\begin{aligned}
(-\Delta)^su(y)=c(N,s)P.V.\int_{\mathbb{R}^N}\frac{u(y)-u(x)}{|x-y|^{N+2s}}dx=c(N,s)\lim\limits_{\epsilon\rightarrow 0^+}\int_{\mathbb{R}^N\setminus B_\epsilon(x)}\frac{u(y)-u(x)}{|x-y|^{N+2s}}dx,
\end{aligned}
\end{equation}
where $P.V.$ is the Cauchy principal value and $c(N,s)$  is a constant depending on $N$ and $s.$ This operator is well defined in $C^{1,1}_{loc}\cap \mathcal {L}_s$, where $\mathcal {L}_s=\{u\in L^1_{loc}:\int_{\mathbb{R}^N}\frac{|u(x)|}{1+|x|^{N+2s}}dx<\infty\}$.  For more details on the fractional Laplacian, we referee to \cite{Eleonora} and the references therein.

The Fractional Laplacian operator appears in dives areas including biological modeling, physics and mathematical finances, and can be regarded as the infinitesimal generator of a stable Levy process (see for example \cite{ad}). From the view point of mathematics, an important feature of the fractional Laplacian operator is its nonlocal property, which makes it more challenge than the classical Laplacian operator. Thus,  problems with the fractional Laplacian have been extensively studied, both for
the pure mathematical research and in view of concrete real-world applications, see for example, \cite{BBarrios}-\cite{XCabre}, \cite{CS2007}, \cite{Quaas}, \cite{Servadeia}, \cite{TJin}, \cite{JTan1}, \cite{JTan}, \cite{SYan} and the references therein.

Solutions of \eqref{problem1} are related to the existence of standing wave solutions to the following fractional Schr\"{o}dinger equation
\begin{equation}
\left\{\begin{array}{ll}
i\partial_t\Psi+(-\Delta)^s\Psi=F(x,\Psi) \ \hbox{ in } \ \mathbb{R}^N, \\
\lim\limits_{|x|\rightarrow\infty}|\Psi(x,t)|=0, \ \hbox{for \ all} \ t>0.
\end{array}\right.
\end{equation}
That is, solutions with the form $\Psi(x,t)=e^{-ict}u(x)$, where c is a constant.

In this paper, under a weaker symmetry condition for $V(y)$,  we will construct multi-bump solutions of \eqref{problem1} through a  finite dimensional reduction method combining with various Pohozaev identies. More precisely,  we consider the case $V(y)=V(|y'|,y'')=V(r,y'')$, $y=(y',y'')\in \mathbb{R}^2\times\mathbb{R}^{N-2}$ and assume that:

\vskip8pt

{\it{($V$) \  Let $V(y)\geq0$ is a bounded function and belongs to $C^2(\mathbb{R}^N)$. Suppose that $r^{2s}V(r,y'')$ has a critical point $(r_0,y_0'')$ satisfying $r_0 > 0$, $ V(r_0,y_0'') > 0$ and
$deg \big(\nabla \big(r^{2s}V(r,y'')\big), (r_0,y_0'')\big) \neq 0.$}}

\vskip8pt

Before the statement of the main results. Let us first introduce some notations.
Denote $D^{s}(\mathbb{R}^N)$ the completion of $C_0^\infty(\mathbb{R}^N)$ under the norm $\|(-\Delta)^{\frac{s}{2}}u\|_{L^2(\mathbb{R}^N)}$, where $\|(-\Delta)^{\frac{s}{2}}u\|_{L^2(\mathbb{R}^N)}$ is defined by $(\int_{\mathbb{R}^N}|\xi|^{2s}|\mathcal {G}u(\xi)|^2 d\xi)^{\frac{1}{2}}$, and $\mathcal {G}u$ is the Fourier transformation of $u$:
$$\mathcal {G}u(\xi)=\frac{1}{(2\pi)^{\frac{N}{2}}}\int_{\mathbb{R}^N}e^{-i\xi\cdot x}u(x)dx.$$
We will construct the solutions in  following  the energy space:
$$H^{s}(\mathbb{R}^N)=\{u\in D^{s}(\mathbb{R}^N): \int_{\mathbb{R}^N} V(y)u^2dy<+\infty\}$$
with the norm:
$$\|u\|_{H^{s}(\mathbb{R}^N)}=\big(\|(-\Delta)^{\frac{s}{2}}u\|^2_{L^2(\mathbb{R}^N)}+\int_{\mathbb{R}^N} V(y)u^2dy\big)^{\frac{1}{2}},$$

We define the functional $I$ on $H^{s}(\mathbb{R}^N)$ by:
\begin{equation}\label{func}
I(u)=\frac{1}{2}\int_{\mathbb{R}^N}|(-\Delta)^{\frac{s}{2}}u|^2dy+\frac{1}{2}\int_{\mathbb{R}^N}V(|y'|,y'')u^2dy
-\frac{1}{2^*_s}\int_{\mathbb{R}^N}(u)_+^{2^*_s}dy,
\end{equation}
where $(u)_+=\max(u,0)$. Then the solutions of problem \eqref{problem1} correspond to the critical points of the functional $I.$

It is well known that the following functions $$U_{x,\lambda}(y)=C(N,s)(\frac{\lambda}{1+\lambda^2|y-x|^2})^{\frac{N-2s}{2}}, \ \ \lambda>0, \ x\in \mathbb{R}^N,$$ where $C(N,s)=2^{\frac{N-2s}{2}}\frac{\Gamma(\frac{N+2s}{2})}{\Gamma(\frac{N-2s}{2})}$,
are the only solutions for the problem (see \cite{{ELieb}}):
\begin{equation}\label{single}(-\Delta)^s u=u^{\frac{N+2s}{N-2s}}, \ \ u>0 \ \hbox{ in } \ \mathbb{R}^N.\end{equation}

Define
\begin{equation*}
\begin{aligned}
H_s=\{u: & u\in H^s(\mathbb{R}^{N}), u(y_1, y_2, y'')=u(y_1, -y_2, y''), \\
 &u(rcos(\theta+\frac{2\pi j}{k}),rsin(\theta+\frac{2\pi j}{k}),y'')=u(rcos\theta,rsin\theta,y'')\}.
\end{aligned}
\end{equation*}
Let $$x_j=(\overline{r}cos\frac{2(j-1)\pi}{k},\overline{r}sin\frac{2(j-1)\pi}{k},\overline{y}''), \ \ j=1,\ldots,k.$$

To construct the solution of \eqref{problem1}, we hope to use $U_{x,\lambda}(y)$  as an approximation solution. However,  the decay of $U_{x_j,\lambda}$ is not fast enough for us when the dimension  $N\leq 6s$. So, we need to cut off this function. Let $\delta>0$ be a small constant such that $r^{2s}V(r,y'')>0$ if $|(r,y'')-(r_0,y''_0)|\leq10\delta$. Take $\zeta(y)=\zeta(r,y'')$ be a $C^2$ smooth function satisfying $\zeta=1$ if $|(r,y'')-(r_0, y''_0)|\leq \delta$, $\zeta=0$ if $|(r,y'')-(r_0, y''_0)|\geq 2\delta$, $|\nabla\zeta|\leq C$ and $0\leq\zeta\leq1$.
Denote
$$Z_{x_j,\lambda}=\zeta U_{x_j,\lambda}, \ Z_{\overline{r}, \overline{y}'',\lambda}^*=\sum_{j=1}^kU_{x_j,\lambda}, \  Z_{\overline{r}, \overline{y}'',\lambda}=\sum_{j=1}^kZ_{x_j,\lambda}.$$
But this will bring us a new difficult. In the proof of Proposition \ref{proposition2}, we have to deal with $(-\Delta)^s(\zeta(y)U_{x_j,\lambda}(y))$. By \eqref{aum23}, we can deduce that
$$(-\Delta)^s\big(\zeta(y)U_{x_j,\lambda}(y)\big)=\zeta(y)U_{x_j,\lambda}^{2_s^*-1}(y)+
c(N,s)\lim\limits_{\epsilon\rightarrow 0^+}\int_{\mathbb{R}^N\setminus B_\epsilon(x)}\frac{\big(\zeta(y)-\zeta(x)\big)U_{x_j,\lambda}(x)}{|x-y|^{N+2s}}dx.$$
In order to obtain a good enough result, we need to calculate the last principal value every carefully (see Lemma \ref{lemma6}).

In the following of the present paper, we always assume that $\tau=\frac{N-4s}{N-2s}$, $k>0$ is a large integer, and $\lambda\in[L_{0}k^{\frac{N-2s}{N-4s}},  L_{1}k^{\frac{N-2s}{N-4s}}]$ for some constants $L_1>L_0>0$,
$|(\overline{r},\overline{y}'')-(r_0, y''_0)|\leq \theta$ with $\theta>0$ is a small constant, and
$$\Omega_j=\{y: y=(y',y'')\in \mathbb{R}^2\times\mathbb{R}^{N-2}, \langle \frac{y'}{|y'|},\frac{x_j'}{|x_j'|}\rangle\geq cos\frac{\pi}{k}\}.$$

Our main result is:
\begin{theorem}\label{th:1}
Suppose that $s\in(\frac{1}{2}, 1)$ and $N>4s+2\tau$. If $V(y)$
satisfies the condition $(V)$, then there is an integer $k_{0}>0$,
such that for any integer $k\geq k_{0}$, problem \eqref{problem1} has
a solution $u_k$ of the form
$$u_{k}=Z_{\overline{r}_k, \overline{y}''_k,\lambda_k}+\phi_{k}$$ where $\phi_{k}\in
H_{s}$, $\lambda_{k}\in[L_{0}k^{\frac{N-2s}{N-4s}}, L_{1}k^{\frac{N-2s}{N-4s}}]$, and as $k\rightarrow\infty$, $\lambda_k^{-\frac{N-2s}{2}}\|\phi_{k}\|_{L^{\infty}}\rightarrow 0$,
$(\overline{r}_k,\overline{y}''_k)\rightarrow(r_0,y''_0).$

\end{theorem}

We will prove Theorem \ref{th:1} by finite dimensional reduction method combining various Pohozaev identities. Finite dimensional reduction method has been extensively used to construct solutions for equations with critical growth. We referee to \cite{bc}, \cite{cny}-\cite{clin2}, \cite{YGuo}, \cite{Guo2016}, \cite{GPY}, \cite{liyy19},
\cite{yylinwm}, \cite{niu},\cite{nyan}, \cite{szhang}, \cite{yan}
and references therein. Roughly speaking, the outline to carry out the reduction argument is as follows: We first construct an good enough approximation solution and linearized the original problem around the approximation solution. Then we solve the corresponding finite dimensional problem to obtain a true solution. To finish the second step, we have to obtain a better estimate for the error term. In this paper, since the operator is nonlocal and the potential function is assumed to have weak symmetry, we have to modify both steps.

\vskip8pt

We introduce the following norms:
$$\|u\|_{*}=\sup\limits_{y\in\mathbb{R}^N}(\sum\limits_{j=1}^k\frac{1}{(1+\lambda|y-x_j|)^{\frac{N-2s}{2}+\tau}})^{-1}\lambda^{-\frac{N-2s}{2}}|u(y)|$$
and
$$\|f\|_{**}=\sup\limits_{y\in\mathbb{R}^N}(\sum\limits_{j=1}^k\frac{1}{(1+\lambda|y-x_j|)^{\frac{N+2s}{2}+\tau}})^{-1}\lambda^{-\frac{N+2s}{2}}|f(y)|.$$

We first use $Z_{\overline{r}, \overline{y}'',\lambda}$ as an approximate solution to obtain a unique function $\phi(\overline{r}, \overline{y}'',\lambda)$, then the problem of finding critical points for $I(u)$ can be reduced to that of finding critical points of $F(\overline{r},\overline{y}'',\lambda)=I(Z_{\overline{r}, \overline{y}'',\lambda}+\phi(\overline{r}, \overline{y}'',\lambda))$. Then in the second step, we solve the corresponding finite dimensional problem to obtain a solution. However, in the first step, we can only obtain $\|\phi\|_*\leq\frac{C}{\lambda^{s+\sigma}}$ (see Proposition \ref{proposition2}).  From Lemma \ref{exp2} and \ref{exp3}, we know that
\begin{equation} \label{energyexpansion11}
\begin{aligned}
\frac{\partial F}{\partial\lambda}=\frac{\partial I(Z_{\overline{r},\overline{y}'',\lambda})}{\partial\lambda}+O(k\lambda^{-1}\|\phi\|^2_*)=k\left(-\frac{B_1}{\la^{2s+1}}V(\bar{r},\bar{y}'')+ \frac{B_3k^{N-2s}}{\la^{N-2s+1}}+O\big(\frac{1}{\la^{2s+1+\sigma}}\big)\right),
\end{aligned}
\end{equation}
\begin{equation} \label{energyexpansion12}
\begin{aligned}
\frac{\partial F}{\partial\overline{r}}
=&\frac{\partial I(Z_{\overline{r},\overline{y}'',\lambda})}{\partial\overline{r}}+O(k\lambda\|\phi\|^2_*)\\
=&k\left(\frac{B_1}{\lambda^{2s}}\frac{\partial V(\overline{r},\overline{y}'')}{\partial\overline{r}}+\sum\limits_{j=2}^k\frac{B_2}{\overline{r}\lambda^{N-2s}|x_1-x_j|^{N-2s}}+O(\frac{1}{\lambda^{s+\sigma}})\right),
\end{aligned}
\end{equation}
and
\begin{equation} \label{energyexpansion13}
\begin{aligned}
\frac{\partial F}{\partial\overline{y}''_j}=\frac{\partial I(Z_{\overline{r},\overline{y}'',\lambda})}{\partial\overline{y}''_j}+O(k\lambda\|\phi\|^2_*)=k\left(\frac{B_1}{\lambda^{2s}}\frac{\partial V(\overline{r},\overline{y}'')}{\partial\overline{y}''_j}+O(\frac{1}{\lambda^{s+\sigma}})\right).
\end{aligned}
\end{equation}
Note that the estimate of $\phi$ is only good enough for the expansion \eqref{energyexpansion11}. But it destroys the main terms in the expansions of  \eqref{energyexpansion12} and \eqref{energyexpansion13}.
To overcome this difficulty, following  the idea in Peng, Wang and Yan \cite{PWY},  instead of studying \eqref{energyexpansion12} and \eqref{energyexpansion13}, we turn to prove that if $(\bar r, \bar{y}'', \lambda)$ satisfies the following local Pohozaev identities:
\begin{eqnarray}\label{firstpohozaevidentity11}
\begin{aligned}
&\quad-\int_{\partial''\mathcal B^{+}_{\rho}}t^{1-2s}\frac{\partial \tilde{u}_{k}}{\partial \nu}\frac{\partial \tilde{u}_{k}}{\partial y_{i}}+\frac{1}{2}\int_{\partial''\mathcal B^{+}_{\rho}}t^{1-2s}|\nabla \tilde{u}_{k}|^{2}\nu_{i}\\
&=\int_{B_{\rho}}\big(-V(r,y'')u_k+(u_k)_+^{2^*_s-1}\big)\frac{\partial u_k}{\partial y_i}, \ \ \ \ \ \ i=3,\ldots, N,
\end{aligned}
\end{eqnarray}
and
\begin{eqnarray}\label{secondtpohozaevidentity11}
\begin{aligned}
&\quad-\int_{\partial''\mathcal B^{+}_{\rho}}t^{1-2s}\langle\nabla \tilde{u}_k, Y\rangle \frac{\partial\tilde{u}_k}{\partial\nu} +\frac{1}{2}\int_{\partial''\mathcal B^{+}_{\rho}}t^{1-2s}|\nabla \tilde{u}_k|^{2}\langle Y,\nu\rangle +\frac{2s-N}{2}\int_{\partial\mathcal B^{+}_{\rho}}t^{1-2s}\frac{\partial\tilde{u}_k}{\partial\nu} \tilde{u}_k\\
&=\int_{ B_{\rho}}\big(-V(r,y'')u_k+(u_k)_+^{2^*_s-1}\big)\langle y, u_k\rangle,
\end{aligned}
\end{eqnarray}
where $u_k=Z_{\overline{r}, \overline{y}'', \lambda}+\phi$, $\tilde{u}_k$  is the extension of $u_k$ (see below \eqref{poisson involution}),
$$ \mathcal B^{+}_{\rho}=\{Y=(y,t):|Y-(r_0,y''_0,0)|\leq\rho \ \hbox{and} \ t>0  \}\subseteq {\Bbb R}_{+}^{N+1},$$
$$ \partial'\mathcal B^{+}_{\rho}=\{Y=(y,t):|y-(r_0,y''_0)|\leq\rho,t=0\}\subseteq {\Bbb R}^{N},$$
$$ \partial''\mathcal B^{+}_{\rho}=\{Y=(y,t):|Y-(r_0,y''_0,0)|=\rho,t>0 \}\subseteq {\Bbb R}_{+}^{N+1},$$
$$\partial\mathcal B^{+}_{\rho}=\partial'\mathcal B^{+}_{\rho}\cup\partial''\mathcal B^{+}_{\rho},$$
$$ B_{\rho}=\{y:|y-(r_0,y''_0)|\leq\rho\}\subseteq {\Bbb R}^{N}.$$

For any $u\in D^{s}(\mathbb{R}^N)$, $\widetilde{u}$ is defined by:
\begin{eqnarray}\label{poisson involution}
 \widetilde{u}(y,t)=\mathcal P_{s}[u]:=\int_{{\Bbb R}^{N}} P_s(y-\xi,t)u(\xi)d\xi,\quad (y,t)\in {\Bbb R}^{N+1}_+:={\Bbb R}^{N}\times (0,+\infty),
\end{eqnarray}
where
\[
 P_s(x,t)=\beta(N,s)\frac{t^{2s}}{(|x|^2+t^2)^{\frac{N+2s}{2}}}
\] with constant $\beta(N,s)$ such that $\int_{{\Bbb R}^{N}}P_s(x,1)d x=1$.
We refer $\widetilde{u}=\mathcal P_{s}[u]$ to be the \emph{extension} of $u$. Moreover, $\widetilde{u}$ satisfies (see \cite{CS2007})
  \begin{eqnarray*}
 \mathrm{div}(t^{1-2s}\nabla \widetilde{u})=0, \quad \mbox{in }\mathbb{R}^{N+1}_+
\end{eqnarray*}
and
\begin{eqnarray*}
-\lim\limits_{t\to 0}t^{1-2s}\pa_t \widetilde{u}(y,t)=\omega_s(-\Delta)^{s} u(y), \quad \mbox{on} \ {\Bbb R}^{N}
\end{eqnarray*}
in the distribution sense, where $\omega_s=2^{1-2s}\Gamma(1-s)/\Gamma(s)$.

Due to the nonlocalness of the fractional Laplacian operator, we have to overcome some serious difficulties. Indeed, we do not have the local Pohozaev identities for $u$,  but for $\widetilde{u}$. The integrals appearing in \eqref{firstpohozaevidentity11} and \eqref{secondtpohozaevidentity11} is much  more complicated. We have to integrate one more time than the Laplacian operator case. It is very difficult when we derive some sharp estimates for each term in \eqref{firstpohozaevidentity11} and \eqref{secondtpohozaevidentity11}. We need a lot of preliminary lemmas.


Our paper is organized as follows. In section 2, we perform a finite dimensional reduction.
We prove the Theorems \ref{th:1} in section 3.
In Appendix A, we give some essential estimates.
We put the energy expansions for $\langle I'(Z_{\overline{r},\overline{y}'',\lambda}+\phi(\overline{r},\overline{y}'',\lambda)),\frac{\partial Z_{\overline{r},\overline{y}'',\lambda}}{\partial\lambda}\rangle$, $\langle I'(Z_{\overline{r},\overline{y}'',\lambda}+\phi(\overline{r},\overline{y}'',\lambda)),\frac{\partial Z_{\overline{r},\overline{y}'',\lambda}}{\partial\overline{r}}\rangle$ and $\langle I'(Z_{\overline{r},\overline{y}'',\lambda}+\phi(\overline{r},\overline{y}'',\lambda)),\frac{\partial Z_{\overline{r},\overline{y}'',\lambda}}{\partial\overline{y}''}\rangle$ in Appendix B.

\section{Finite dimensional reduction}

In this section, we perform a finite dimensional reduction by using $Z_{\overline{r}, \overline{y}'',\lambda}$ as an
approximation solution and considering the linearization of the
problem \eqref{problem1} around the approximation solution
$Z_{\overline{r}, \overline{y}'',\lambda}$.
Let $$Z_{i,1}=\frac{\partial Z_{x_i,\lambda}}{\partial \lambda}, \ \ Z_{i,2}=\frac{\partial Z_{x_i,\lambda}}{\partial \overline{r}}, \ \ Z_{i,k}=\frac{\partial Z_{x_i,\lambda}}{\partial \overline{y}''_k}, \ k=3,\ldots,N.$$
Then direct computation shows that
$$Z_{i,1}=O(\lambda^{-1}Z_{x_i,\lambda}), \ \ Z_{i,l}=O(\lambda Z_{x_i,\lambda}), \ l=2,\ldots,N.$$


We consider the following linearized problem:
\begin{equation}\label{problem3}
\left\{\begin{array}{ll}
(-\Delta)^s \phi+V(r,y'')\phi-(2^*_s-1)Z_{\overline{r},\overline{y}'',\lambda}^{2^*_s-2}\phi=h+\sum\limits_{l=1}^Nc_l\sum\limits_{i=1}^kZ_{x_i,\lambda}^{2^*_s-2}Z_{i,l},\\
 u\in H_{s}, \ \ \sum\limits_{i=1}^k\int_{\mathbb{R}^N}Z_{x_i,\lambda}^{2^*_s-2}Z_{i,l}\phi=0, \ l=1,2,\ldots,N,
\end{array}\right.
\end{equation}
for some numbers $c_l$.

\begin{lemma}\label{lemma4}
Suppose that $N>4s$ and $\phi_k$ solves problem \eqref{problem3}. If $\|h_k\|_{**}\rightarrow0$ as $k\rightarrow\infty$, then $\|\phi_k\|_{*}\rightarrow0$ as $k\rightarrow\infty$.
\end{lemma}
{\bf Proof.}  We prove this lemma
by contradiction arguments. Assume that there exist $h_{k}$
with $\|h_{k}\|_{\ast\ast}\rightarrow0$ as $k\rightarrow\infty$, $\|\phi_k\|_{\ast}\geq c>0$ with $\lambda=\lambda_k,$ $\lambda_{k}\in[L_{0}k^{\frac{N-2s}{N-4s}},L_{1}k^{\frac{N-2s}{N-4s}}]$ and
$(\overline{r}_k,\overline{y}''_k)\rightarrow(r_0,y''_0)$.
Without loss of generality, we can assume that
$\|\phi_{k}\|_{\ast}\equiv1.$  For simplicity, we drop the subscript $k$.

Firstly, we have
\begin{equation}
\begin{aligned}
|\phi(y)|\leq& C\int_{\mathbb{R}^{N}}\frac{1}{|y-z|^{N-2s}}Z_{\overline{r},\overline{y}'',\lambda}^{2^*_s-2}|\phi|dz\\
&+C\int_{\mathbb{R}^{N}}\frac{1}{|y-z|^{N-2s}}\Big[|h|+|\sum\limits_{l=1}^{N}c_l\sum\limits_{i=1}^{k}
Z_{x_{i},\lambda}^{2^*_s-2}Z_{i,l}|\Big]dz\\
=:&A_1+A_2.
\end{aligned}
\end{equation}

For the first term $A_1$, by Lemma \ref{Lemma appendix1} and \ref{Lemma appendix2}, we can deduce that
\begin{equation}\label{equality6}
\begin{aligned}
\big|A_1\big|\leq&C\|\phi\|_{\ast}\int_{\mathbb{R}^{N}}\frac{1}{|y-z|^{N-2s}}Z_{\overline{r},\overline{y}'',\lambda}^{2^*_s-2}
\sum\limits_{i=1}^{k}\frac{\lambda^{\frac{N-2s}{2}}}{(1+\lambda|z-x_{i}|)^{\frac{N-2s}{2}+\tau}}dz\\
\leq&C\|\phi\|_{\ast}\lambda^{\frac{N-2s}{2}}\sum\limits_{i=1}^{k}\frac{1}{(1+\lambda|y-x_{i}|)^{\frac{N-2s}{2}+\tau+\theta}},
\end{aligned}
\end{equation}
where $\theta$ is a small constant. For the second term $A_2$, we make use of Lemma \ref{Lemma appendix2}, so that
\begin{equation}\label{equality7}
\begin{aligned}
 \big|A_2\big|&\leq
C\|h\|_{\ast\ast}\int_{\mathbb{R}^{N}}\sum\limits_{i=1}^{k}
\frac{\lambda^{\frac{N+2s}{2}}}{|y-z|^{N-2s}(1+\lambda|z-x_{i}|)^{\frac{N+2s}{2}+\tau}}dz\\
&\quad +C\sum\limits_{l=1}^{N}|c_l|\int_{\mathbb{R}^{N}}\sum\limits_{i=1}^{k}\frac{\lambda^{\frac{N+2s}{2}+n_l}}{|y-z|^{N-2s}(1+\lambda|z-x_{i}|)^{N+2s}}dz\\
&\leq
C\|h\|_{\ast\ast}\lambda^{\frac{N-2s}{2}}\sum\limits_{i=1}^{k}\frac{1}{(1+\lambda|y-x_{i}|)^{\frac{N-2s}{2}+\tau}}
+C\sum\limits_{l=1}^{N}|c_l|\lambda^{\frac{N-2s}{2}+n_l}\sum\limits_{i=1}^{k}\frac{1}{(1+\lambda|y-x_{i}|)^{\frac{N-2s}{2}+\tau}},
\end{aligned}
\end{equation}
where $n_1=-1$, $n_l=1$ for $l=2,\ldots,N$.
Then, we have
\begin{equation}\label{equality13}
\begin{aligned}
&\big(\sum\limits_{i=1}^{k}\frac{1}{(1+\lambda|y-x_{i}|)^{\frac{N-2s}{2}+\tau}}\big)^{-1}\lambda^{-\frac{N-2s}{2}}|\phi|
\\ \leq & C\|\phi\|_{\ast}\frac{\sum\limits_{i=1}^{k}\frac{1}{(1+\lambda|y-x_{i}|)^{\frac{N-2s}{2}+\tau+\theta}}}{\sum\limits_{i=1}^{k}\frac{1}{(1+\lambda|y-x_{i}|)^{\frac{N-2s}{2}+\tau}}}
+C\|h\|_{\ast\ast}+C\sum\limits_{l=1}^{N}|c_l|\lambda^{n_l}.
\end{aligned}
\end{equation}

Multiplying both sides of \eqref{problem3} by $Z_{1,t}$, we have
\begin{equation}\label{equality8}
\begin{aligned}
&\sum\limits_{l=1}^{N}c_l\sum\limits_{i=1}^{k}\int_{\mathbb{R}^N}Z_{x_i,\lambda}^{2^*_s-2}Z_{i,l}Z_{1,t}\\
=&\big\langle(-\Delta )^s\phi-V(r,y'')\phi-(2^*_s-1)Z^{2^*_s-2}_{\overline{r},\overline{y}'',\lambda}\phi,Z_{1,t}\big\rangle-\langle h,Z_{1,t}\rangle.
\end{aligned}
\end{equation}
First of all, there exists a constant $\overline{c}>0$ such that
\begin{equation}\label{equality11}
\begin{aligned}
\sum\limits_{i=1}^{k}\int_{\mathbb{R}^N}Z_{x_i,\lambda}^{2^*_s-2}Z_{i,l}Z_{1,t} \left\{
\begin{array}{ll}
=(\overline{c}+o(1))\lambda^{2n_t}, & l=t, \\
\leq \frac{\overline{c}\lambda^{n_t}\lambda^{n_l}}{\lambda^{N}}, & l\neq t.
\end{array}
   \right.
\end{aligned}
\end{equation}
On the other hand, we have
\begin{equation}\label{aum16}
\begin{aligned}
&|\langle V(r,y'')\phi,Z_{1,t}\rangle|\\
\leq& C\|\phi\|_*\int_{\mathbb{R}^N}\frac{\zeta\lambda^{N-2s+n_t}}{(1+\lambda|y-x_1|)^{N-2s}}
\sum\limits_{i=1}^{k}\frac{1}{(1+\lambda|y-x_i|)^{\frac{N-2s}{2}+\tau}}\\
\leq& C\|\phi\|_*\lambda^{N-2s+n_t}\big[\int_{\mathbb{R}^N}\frac{\zeta}{(1+\lambda|y-x_1|)^{\frac{3N-6s}{2}+\tau}}\\
&\quad \quad \quad \quad +\sum\limits_{i=2}^{k}\frac{1}{(\lambda|x_1-x_i|)^{\tau}}\int_{\mathbb{R}^N}\zeta(\frac{1}{(1+\lambda|y-x_1|)^{\frac{3N-6s}{2}}}+\frac{1}{(1+\lambda|y-x_i|)^{\frac{3N-6s}{2}}})\big]\\
\leq&C\|\phi\|_*\int_{\mathbb{R}^N}\zeta\frac{\lambda^{N-2s+n_t}}{(1+\lambda|y-x_1|)^{\frac{3N-6s}{2}}}\leq\frac{C\lambda^{n_t}\|\phi\|_*log\lambda}{\lambda^{\min(2s,\frac{N-2s}{2})}}
\leq\frac{C\lambda^{n_t}\|\phi\|_*}{\lambda^{s+\sigma}}
\end{aligned}
\end{equation}
and
\begin{equation}\label{aum17}
\begin{aligned}
|\langle h,Z_{1,t}\rangle|&\leq C\|h\|_{\ast\ast}\int_{\mathbb{R}^{N}}\frac{\lambda^{N+n_t}}{(1+\lambda|y-x_1|)^{N-2s}}\sum\limits_{i=1}^{k}\frac{1}{(1+\lambda|y-x_i|)^{\frac{N+2s}{2}+\tau}}\\
&\leq C\lambda^{n_t}\|h\|_{\ast\ast}.
\end{aligned}
\end{equation}

Moreover, one has
\begin{equation}\label{equality12}
\begin{aligned}
|\langle(-\Delta)^s\phi-(2^*_s-1)Z^{2^*_s-2}_{\overline{r},\overline{y}'',\lambda}\phi,Z_{1,t}\rangle|\leq\frac{C\lambda^{n_t}\|\phi\|_{\ast}}{\lambda^{s+\sigma}}.
\end{aligned}
\end{equation}

Combining \eqref{equality8}, \eqref{equality11}, \eqref{aum16}, \eqref{aum17} and \eqref{equality12}, we have
$$|c_t|\leq \frac{C}{\lambda^{n_t}}(\frac{\|\phi\|_{\ast}}{\lambda^{\sigma}}+\|h\|_{\ast\ast})+\frac{C}{\lambda^{n_t}}\sum\limits_{l\neq t}\frac{\lambda^{n_l}|c_l|}{\lambda^{N}}.$$
This implies that
$$\sum\limits_{l=1}^N|c_l|\lambda^{n_l}\leq C(\frac{\|\phi\|_{\ast}}{\lambda^{\sigma}}+\|h\|_{\ast\ast}).$$

Thus by \eqref{equality13} and $\|\phi\|_{\ast}=1$, there is $R>0$ such that
\begin{equation}\label{equality20}
\begin{aligned}
\|\lambda^{-\frac{N-2s}{2}}\phi(y)\|_{L^\infty(B_{R/\lambda}(x_i))}\geq a>0,
\end{aligned}
\end{equation}
for some $i$. As a result, we have that $\widetilde{\phi}=\lambda^{-\frac{N-2s}{2}}\phi(\frac{y}{\lambda}+x_i)$ converges uniformly, in any compact set, to a solution $u$ of the following equation:
$$(-\Delta)^su-(2^*_s-1)U_{0,\Lambda}^{2^*_s-2}u=0, \ \ \hbox{ in } \ \mathbb{R}^N,$$
for some $0<\Lambda_1\leq\Lambda\leq\Lambda_2$. Since $u$ is perpendicular to the kernel of this equation, $u=0$. This is a contradiction to \eqref{equality20}.
\hfill{$\Box$}

\vskip8pt

Using the same argument as in the proof of Proposition 4.1 in \cite{pfm}, we can obtain the following proposition.
\begin{proposition}\label{proposition1}
There exist $k_0>0$ and a constant $C>0$, independent of $k$, such that for all
$k\geq k_0$ and all $h\in L^\infty(\mathbb{R}^N)$, problem \eqref{problem3} has a unique solution $\phi=L_k(h)$. Besides,
\begin{equation}\label{equality1}
\begin{aligned}
\|L_k(h)\|_*\leq C\|h\|_{**}, \ \ \ |c_l|\leq \frac{C}{\lambda^{n_l}}\|h\|_{**}.
\end{aligned}
\end{equation}
\end{proposition}

Now we consider the following problem:
\begin{equation}\label{problem4}
\left\{\begin{array}{ll}
(-\Delta)^s(Z_{\overline{r},\overline{y}'',\lambda}+\phi)+V(r,y'')(Z_{\overline{r},\overline{y}'',\lambda}+\phi)=(Z_{\overline{r},\overline{y}'',\lambda}+\phi)^{2_s^*-1}+\sum\limits_{l=1}^Nc_l\sum\limits_{i=1}^kZ^{2_s^*-2}_{x_i,\lambda}Z_{i,l}, \ \ \hbox{ in } \mathbb{R}^N, \\
\phi\in H_{s},
\sum\limits_{i=1}^k\int_{\mathbb{R}^N}Z^{2_s^*-2}_{x_i,\lambda}Z_{i,l}\phi=0, \ l=1,\ldots, N.
\end{array}\right.
\end{equation}
In the rest of this section, we devote
ourselves to the proof of  the following proposition by using the contraction mapping theorem.
\begin{proposition}\label{proposition2}
There exist $k_0>0$ and a constant $C>0$, independent of $k$, such that for all
$k\geq k_0$, $L_0k^{\frac{N-2s}{N-4s}}\leq \lambda\leq L_1k^{\frac{N-2s}{N-4s}}$, $|(\overline{r},\overline{y}'')-(r_0, y''_0)|\leq \theta$, problem \eqref{problem4} has a unique solution $\phi=\phi(\overline{r},\overline{y}'',\lambda)$ satisfying,
\begin{equation}\label{equality2}
\begin{aligned}
\|\phi\|_*\leq C(\frac{1}{\lambda})^{s+\sigma}, \ \ \ |c_l|\leq C(\frac{1}{\lambda})^{s+\sigma},
\end{aligned}
\end{equation}
where $\sigma>0$ is a small constant.
\end{proposition}

We rewrite \eqref{problem4} as
\begin{equation}\label{problem5}
\left\{\begin{array}{ll}
(-\Delta)^s\phi+V(r,y'')\phi-(2_s^*-1)(Z_{\overline{r},\overline{y}'',\lambda})^{2_s^*-2}\phi=\mathcal {F}(\phi)+l_k+\sum\limits_{l=1}^Nc_l\sum\limits_{i=1}^kZ^{2_s^*-2}_{x_i,\lambda}Z_{i,l}, \ \ \hbox{ in } \mathbb{R}^N, \\
\phi\in H_{s},
\sum\limits_{i=1}^k\int_{\mathbb{R}^N}Z^{2_s^*-2}_{x_i,\lambda}Z_{i,l}\phi=0, \ l=1,\ldots, N.
\end{array}\right.
\end{equation}
where
$$\mathcal{F}(\phi)=(Z_{\overline{r},\overline{y}'',\lambda}+\phi)^{2_s^*-1}_+-Z_{\overline{r},\overline{y}'',\lambda}^{2_s^*-1}-(2_s^*-1)Z_{\overline{r},\overline{y}'',\lambda}^{2_s^*-2}\phi,$$
and
\begin{equation}
\begin{aligned}
l_k(y)&=\big(Z_{\overline{r},\overline{y}'',\lambda}^{2_s^*-1}(y)-\zeta(y) \sum\limits_{j=1}^kU_{x_j,\lambda}^{2_s^*-1}(y)\big)-V(r,y'')Z_{\overline{r},\overline{y}'',\lambda}(y)\\
&\quad-\sum\limits_{j=1}^kc(N,s)\lim\limits_{\epsilon\rightarrow 0^+}\int_{\mathbb{R}^N\setminus B_\epsilon(x)}\frac{\big(\zeta(y)-\zeta(x)\big)U_{x_j,\lambda}(x)}{|x-y|^{N+2s}}dx\\
&=:J_1+J_2+J_3.
\end{aligned}
\end{equation}
In order to use the contraction mapping theorem to prove Proposition \ref{proposition2}, we need to estimate $\mathcal{F}(\phi)$ and $l_k$. In the following, we assume that $\|\phi\|_{\ast}$ is small.

\begin{lemma}\label{lemma5}
We have $\|\mathcal{F}(\phi)\|_{**}\leq C\|\phi\|_*^{\min(2,2_s^*-1)}$.
\end{lemma}
{\bf Proof.}
Firstly, we consider $2_s^*\leq3$.
Using the H\"{o}lder inequality, we obtain:
\begin{equation*}
\begin{aligned}
|\mathcal{F}(\phi)|&\leq C\|\phi\|_*^{2_s^*-1}(\sum\limits_{j=1}^k\frac{\lambda^{\frac{N-2s}{2}}}{(1+\lambda|y-x_j|)^{\frac{N-2s}{2}+\tau}})^{2_s^*-1}\\
&\leq C\|\phi\|_*^{2_s^*-1}\lambda^{\frac{N+2s}{2}}\sum\limits_{j=1}^k\frac{1}{(1+\lambda|y-x_j|)^{\frac{N+2s}{2}+\tau}}(\sum\limits_{j=1}^k\frac{1}{(1+\lambda|y-x_j|)^{\tau}})^{\frac{4s}{N-2s}}\\
&\leq C\|\phi\|_*^{2_s^*-1}\lambda^{\frac{N+2s}{2}}\sum\limits_{j=1}^k\frac{1}{(1+\lambda|y-x_j|)^{\frac{N+2s}{2}+\tau}}.
\end{aligned}
\end{equation*}
When $2_s^*>3$, we have
\begin{equation*}
\begin{aligned}
|\mathcal{F}(\phi)|&\leq C\|\phi\|_*^{2}(\sum\limits_{j=1}^k\frac{\lambda^{\frac{N-2s}{2}}}{(1+\lambda|y-x_j|)^{\frac{N-2s}{2}+\tau}})^{2}(\sum\limits_{j=1}^k\frac{\lambda^{\frac{N-2s}{2}}}{(1+\lambda|y-x_j|)^{N-2s}})^{2_s^*-3}\\
&\quad+C\|\phi\|_*^{2_s^*-1}(\sum\limits_{j=1}^k\frac{\lambda^{\frac{N-2s}{2}}}{(1+\lambda|y-x_j|)^{\frac{N-2s}{2}+\tau}})^{2_s^*-1}\\
&\leq C(\|\phi\|_*^{2}+\|\phi\|_*^{2_s^*-1})\lambda^{\frac{N+2s}{2}}(\sum\limits_{j=1}^k\frac{1}{(1+\lambda|y-x_j|)^{\frac{N-2s}{2}+\tau}})^{2_s^*-1}\\
&\leq C\|\phi\|_*^{2}\lambda^{\frac{N+2s}{2}}\sum\limits_{j=1}^k\frac{1}{(1+\lambda|y-x_j|)^{\frac{N+2s}{2}+\tau}}.
\end{aligned}
\end{equation*}
Hence, we obtain
$\|\mathcal{F}(\phi)\|_{**}\leq C\|\phi\|_*^{\min(2,2_s^*-1)}$. \par
\hfill{$\Box$}

\vskip8pt

Next, we estimate $l_k$.
\begin{lemma}\label{lemma6}
If $N>4s+2\tau$, then there exists a small $\sigma>0$ such that  $\|l_k\|_{**}\leq \frac{C}{\lambda^{s+\sigma}}$.
\end{lemma}
{\bf Proof.}
By the symmetry, we can assume that $y\in\Omega_1$. Then $|y-x_j|\geq|y-x_1|$.  We first estimate the term $J_1$. We have
\begin{equation*}
\begin{aligned}
|J_1|&\leq C[(\sum\limits_{j=2}^kU_{x_j,\lambda})^{2_s^*-1}+U_{x_1,\lambda}^{2_s^*-2}\sum\limits_{j=2}^kU_{x_j,\lambda}+\sum\limits_{j=2}^kU_{x_j,\lambda}^{2_s^*-1}]\\
&\leq C\lambda^{\frac{N+2s}{2}}\big(\sum\limits_{j=2}^k\frac{1}{(1+\lambda|y-x_j|)^{N-2s}}\big)^{2_s^*-1}+\frac{C\lambda^{\frac{N+2s}{2}}}{(1+\lambda|y-x_1|)^{4s}}\sum\limits_{j=2}^k\frac{1}{(1+\lambda|y-x_j|)^{N-2s}}.
\end{aligned}
\end{equation*}
If $N-2s\geq\frac{N+2s}{2}-\tau,$ then we have
\begin{equation*}
\begin{aligned}
&\quad \frac{1}{(1+\lambda|y-x_1|)^{4s}}\sum\limits_{j=2}^k\frac{1}{(1+\lambda|y-x_j|)^{N-2s}}\\
&\leq\frac{1}{(1+\lambda|y-x_1|)^{\frac{N+2s}{2}+\tau}}\sum\limits_{j=2}^k\frac{1}{(1+\lambda|y-x_j|)^{\frac{N+2s}{2}-\tau}}\\
&\leq\frac{1}{(1+\lambda|y-x_1|)^{\frac{N+2s}{2}+\tau}}\sum\limits_{j=2}^k\frac{1}{(\lambda|x_1-x_j|)^{\frac{N+2s}{2}-\tau}}\\
&\leq\frac{1}{(1+\lambda|y-x_1|)^{\frac{N+2s}{2}+\tau}}(\frac{k}{\lambda})^{\frac{N+2s}{2}-\tau}.
\end{aligned}
\end{equation*}
If $N-2s<\frac{N+2s}{2}-\tau,$ then $4s>\frac{N+2s}{2}+\tau,$ we obtain that
\begin{equation*}
\begin{aligned}
&\quad \frac{1}{(1+\lambda|y-x_1|)^{4s}}\sum\limits_{j=2}^k\frac{1}{(1+\lambda|y-x_j|)^{N-2s}}\\
&\leq\frac{1}{(1+\lambda|y-x_1|)^{\frac{N+2s}{2}+\tau}}\sum\limits_{j=2}^k\frac{1}{(\lambda|x_1-x_j|)^{N-2s}}\\
&\leq\frac{1}{(1+\lambda|y-x_1|)^{\frac{N+2s}{2}+\tau}}(\frac{k}{\lambda})^{N-2s}.
\end{aligned}
\end{equation*}

Using the H\"{o}lder inequality, we have
\begin{equation*}
\begin{aligned}
&\quad\big(\sum\limits_{j=2}^k\frac{1}{(1+\lambda|y-x_j|)^{N-2s}}\big)^{2_s^*-1}\\
&\leq \sum\limits_{j=2}^k\frac{1}{(1+\lambda|y-x_j|)^{\frac{N+2s}{2}+\tau}}\big(\sum\limits_{j=2}^k\frac{1}{(1+\lambda|y-x_j|)^{\frac{N+2s}{4s}(\frac{N-2s}{2}-\frac{N-2s}{N+2s}\tau)}}\big)^{\frac{4s}{N-2s}}\\
&\leq C\sum\limits_{j=2}^k\frac{1}{(1+\lambda|y-x_j|)^{\frac{N+2s}{2}+\tau}}(\frac{k}{\lambda})^{\frac{N+2s}{N-2s}(\frac{N-2s}{2}-\frac{N-2s}{N+2s}\tau)}\\
&\leq C\sum\limits_{j=2}^k\frac{1}{(1+\lambda|y-x_j|)^{\frac{N+2s}{2}+\tau}}(\frac{1}{\lambda})^{s+\sigma}.
\end{aligned}
\end{equation*}
Thus
$$\|J_1\|_{**}\leq C(\frac{1}{\lambda})^{s+\sigma}.$$

Now, we estimate $J_2$. Note that $\zeta=0$ when $|(r,y'')-(r_0,y''_0)|\geq2\delta$ and
$\frac{1}{\lambda}\leq\frac{C}{1+\lambda|y-x_j|}$ when $|(r,y'')-(r_0,y''_0)|<2\delta$.
We have
\begin{equation*}
\begin{aligned}
|J_2|&\leq\frac{C}{\lambda^{2s}}\lambda^{\frac{N+2s}{2}}\sum\limits_{j=1}^k\frac{\zeta}{(1+\lambda|y-x_j|)^{N-2s}}\\
&\leq\frac{C}{\lambda^{\min(2s, N-\frac{N+2s}{2}-\tau)}}\lambda^{\frac{N+2s}{2}}\sum\limits_{j=1}^k\frac{1}{(1+\lambda|y-x_j|)^{\frac{N+2s}{2}+\tau}}.
\end{aligned}
\end{equation*}
If $N>4s+2\tau$, then $\|J_2\|_{**}\leq\frac{C}{\lambda^{s+\sigma}}$.

\vskip8pt

We have
\begin{equation*}
\begin{aligned}
J_3&=\sum\limits_{j=1}^kc(N,s)\big(\lim\limits_{\epsilon\rightarrow 0^+}\int_{B_{\frac{\delta}{4}}(y)\setminus B_\epsilon(y)}\frac{\big(\zeta(y)-\zeta(x)\big)U_{x_j,\lambda}(x)}{|x-y|^{N+2s}}dx\\
&\quad\quad\quad\quad\quad\quad\quad+\int_{\mathbb{R}^N\setminus B_{\frac{\delta}{4}}(y)}\frac{\big(\zeta(y)-\zeta(x)\big)U_{x_j,\lambda}(x)}{|x-y|^{N+2s}}dx\big)\\
&=:\sum\limits_{j=1}^kc(N,s)(J_{31}+J_{32}).
\end{aligned}
\end{equation*}

We first estimate $J_{31}$. From the definition of function $\zeta$, we have
$\zeta(y)-\zeta(x)=0$ when $x,y\in B_{\delta}(x_j)$ or $x,y\in \mathbb{R}^N\setminus \overline{B_{2\delta}(x_j)}$. So, $J_{31}\neq0$ only when $B_{\frac{\delta}{4}}(y)\subset B_{\frac{5}{2}\delta}(x_j)\setminus B_{\frac{1}{2}\delta}(x_j)$. It holds $\frac{3}{4}\delta\leq|y-x_j|\leq|x-y|+|x-x_j|\leq\frac{\delta}{4}+|x-x_j|\leq\frac{3}{2}|x-x_j|\leq\frac{15}{4}\delta$ when $B_{\frac{\delta}{4}}(y)\subset B_{\frac{5}{2}\delta}(x_j)\setminus B_{\frac{1}{2}\delta}(x_j)$.
Furthermore, we divide $J_{31}$ as following,
\begin{equation*}
\begin{aligned}
J_{31}&=
\lim\limits_{\epsilon\rightarrow 0^+}\int_{B_{\frac{\delta}{4}}(y)\setminus B_\epsilon(y)}\frac{\nabla\zeta(y)\cdot (y-x)U_{x_j,\lambda}(x)}{|x-y|^{N+2s}}dx+O(\lim\limits_{\epsilon\rightarrow 0^+}\int_{B_{\frac{\delta}{4}}(y)\setminus B_\epsilon(y)}\frac{U_{x_j,\lambda}(x)}{|x-y|^{N+2s-2}}dx)\\
&=:J_{311}+J_{312}
\end{aligned}
\end{equation*}
Note that $B_{\frac{\delta}{4}}(y)\setminus B_\epsilon(y)$ is symmetrical set. Then by the mean value theorem, we get that
\begin{equation*}
\begin{aligned}
|J_{311}|&=\big|\lim\limits_{\epsilon\rightarrow 0^+}\int_{B_{\frac{\delta}{4}}(y)\setminus B_\epsilon(y)}\frac{\nabla\zeta(y)\cdot (y-x)U_{x_j,\lambda}(x)}{|x-y|^{N+2s}}\big|\\
&=\big|C(N,s)\lim\limits_{\epsilon\rightarrow 0^+}\int_{B_{\frac{\delta}{4}}(0)\setminus B_\epsilon(0)}\frac{\nabla\zeta(y)\cdot z}{|z|^{N+2s}}\frac{\lambda^{\frac{N-2s}{2}}}{(1+\lambda^2|z+y-x_j|^2)^{\frac{N-2s}{2}}}\big|\\
&=\big|\frac{C(N,s)\lambda^{\frac{N-2s}{2}}}{2}\lim\limits_{\epsilon\rightarrow 0^+}\int_{B_{\frac{\delta}{4}}(0)\setminus B_\epsilon(0)}\frac{\nabla\zeta(y)\cdot z}{|z|^{N+2s}}\\
&\quad\quad\quad\quad\quad\quad\quad\quad\quad\quad\quad\times\big(\frac{1}{(1+\lambda^2|z+y-x_j|^2)^{\frac{N-2s}{2}}}-\frac{1}{(1+\lambda^2|-z+y-x_j|^2)^{\frac{N-2s}{2}}}\big)\big|\\
&\leq C\lambda^{\frac{N-2s}{2}+1}\int_{B_{\frac{\delta}{4}}(0)}\frac{|\nabla\zeta(y)|}{|z|^{N+2s-2}}
\frac{1}{(1+\lambda|(2\vartheta-1)z+y-x_j|)^{N-2s+1}}\\
&\leq\frac{C}{\lambda^{s+\sigma}}\lambda^{\frac{N+2s}{2}}\frac{1}{(1+\lambda|y-x_j|)^{\frac{N+2s}{2}+\tau}},
\end{aligned}
\end{equation*}
where $0<\vartheta<1$ and since $|(2\vartheta-1)z+y-x_j|\geq|y-x_j|-|(2\vartheta-1)z|\geq\frac{2}{3}|y-x_j|$ when $z\in B_{\frac{\delta}{4}}(0)$.
Similarly, we can obtain $$|J_{312}|\leq\frac{C}{\lambda^{s+\sigma}}\lambda^{\frac{N+2s}{2}}\frac{1}{(1+\lambda|y-x_j|)^{\frac{N+2s}{2}+\tau}}.$$

For the term $J_{32}$, we divide three cases:\\
Case1: If $y\in B_\delta(x_j)$,then
\begin{equation*}
\begin{aligned}
|J_{32}|&\leq\int_{\mathbb{R}^N\setminus \big(B_{\frac{\delta}{4}}(y)\cup B_{\delta}(x_j)\big)}\frac{1}{|x-y|^{N+2s}}\frac{\lambda^{\frac{N-2s}{2}}}{(1+\lambda|x-x_j|)^{N-2s}}\\
&\leq\frac{C}{\lambda^{2s}}\lambda^{\frac{N+2s}{2}}\frac{1}{(1+\lambda|y-x_j|)^{N-2s}}\int_{\mathbb{R}^N\setminus B_{\frac{\delta}{4}}(y)}\frac{1}{|x-y|^{N+2s}}\\
&\leq\frac{C}{\lambda^{s+\sigma}}\lambda^{\frac{N+2s}{2}}\frac{1}{(1+\lambda|y-x_j|)^{\frac{N+2s}{2}+\tau}}.
\end{aligned}
\end{equation*}
Case2: If $\delta\leq|y-x_j|\leq3\delta$, then by Lemma \ref{Lemma appendixaA30},
\begin{equation*}
\begin{aligned}
|J_{32}|&\leq\int_{\mathbb{R}^N\setminus B_{\frac{\delta}{4}}(y)}\frac{1}{|x-y|^{N+2s}}\frac{\lambda^{\frac{N-2s}{2}}}{(1+\lambda|x-x_j|)^{N-2s}}\\
&\leq C\lambda^{\frac{N+2s}{2}}\int_{\mathbb{R}^N\setminus B_{\frac{\delta\lambda}{4}}(\lambda y)}\frac{1}{|z-\lambda y|^{N+2s}}\frac{1}{(1+|z-\lambda x_j|)^{N-2s}}\\
&\leq C\lambda^{\frac{N+2s}{2}}\big(\frac{1}{(\lambda|y-x_j|)^N}+\frac{1}{\lambda^{2s}}\frac{1}{(\lambda|y-x_j|)^{N-2s}}\big)\\
&\leq\frac{C}{\lambda^{s+\sigma}}\lambda^{\frac{N+2s}{2}}\frac{1}{(1+\lambda|y-x_j|)^{\frac{N+2s}{2}+\tau}}.
\end{aligned}
\end{equation*}
Case3: Suppose that $|y-x_j|>3\delta$. Note that $|x-y|\geq|y-x_j|-|x-x_j|\geq\frac{1}{3}|y-x_j|$ when $|y-x_j|\geq3\delta$ and $|x-x_j|\leq2\delta$. Then we have
\begin{equation*}
\begin{aligned}
|J_{32}|&\leq\int_{B_{2\delta}(x_j)}\frac{1}{|x-y|^{N+2s}}\frac{\lambda^{\frac{N-2s}{2}}}{(1+\lambda|x-x_j|)^{N-2s}}\\
&\leq \frac{C}{\lambda^{\frac{N-2s}{2}}}\int_{B_{2\delta}(x_j)}\frac{1}{|x-y|^{N+2s}}\frac{1}{|x-x_j|^{N-2s}}\\
&\leq \frac{C\lambda^{\frac{N+2s}{2}}}{\lambda^{N}}\frac{1}{|y-x_j|^{\frac{N+2s}{2}+\tau}}\int_{B_{2\delta}(x_j)}\frac{1}{|x-x_j|^{N-2s}}\\
&\leq\frac{C}{\lambda^{s+\sigma}}\lambda^{\frac{N+2s}{2}}\frac{1}{(1+\lambda|y-x_j|)^{\frac{N+2s}{2}+\tau}}.
\end{aligned}
\end{equation*}

So, we obtain $\|J_3\|_{**}\leq\frac{C}{\lambda^{s+\sigma}}$.

\vskip8pt

As a result, we have proved that $\|l_k\|_{**}\leq\frac{C}{\lambda^{s+\sigma}}$.
\hfill{$\Box$}


\vskip8pt

 {\bf Proof of Proposition \ref{proposition2}.}
 Let $y=(y', y'')$, $y'\in \mathbb{R}^2$, $y''\in \mathbb{R}^{N-2}$. Set
$$\begin{array}{ll}
E=\{u : u\in C(\mathbb{R}^N)\cap H_s, \|u\|_{\ast}\leq\frac{1}{\lambda^s}, \sum\limits_{i=1}^k\int_{\mathbb{R}^N}Z^{2_s^*-2}_{x_i,\lambda}Z_{i,l}u=0, \ l=1,\ldots, N\}.\end{array}$$

By Proposition \ref{proposition1}, the solution
$\phi$ of \eqref{problem4} is equivalent to the following fixed point
problem:
$$\phi=A(\phi)=:L_{k}\big(\mathcal{F}(\phi)\big)+L_{k}\big(l_{k}\big).$$
Hence, it is sufficient to prove that the operator $A$ is a
contraction map from the complete space $E$ to itself. In fact,  for any $\phi\in E$, by Proposition \ref{proposition1}, Lemma
\ref{lemma5} and Lemma \ref{lemma6}, we have
\begin{displaymath}
\begin{aligned}
\|A(\phi)\|_\ast\leq&C\|L_{k}\big(\mathcal{F}(\phi)\big)\|_{\ast}+C\|L_{k}(l_{k})\|_{\ast}\\
\leq&C\Big[\|\mathcal{F}(\phi)\|_{\ast\ast}+\|l_{k}\|_{\ast\ast}\Big]\\
\leq&\frac{1}{\lambda^{s}},
\end{aligned}
\end{displaymath}
which shows that $A$ maps $E$ to $E$ itself and $E$ is invariant
under $A$ operator.

If $2_s^*\leq3$, then for $\forall \phi_1, \phi_2\in E$, we have
\begin{displaymath}
\begin{aligned}
\|A(\phi_{1})-A(\phi_{2})\|_*=&\|L_{k}\big(\mathcal {F}(\phi_{1})-\mathcal {F}(\phi_{2})\big)\|_{\ast}\\
\leq&C\|\mathcal {F}(\phi_{1})-\mathcal {F}(\phi_{2})\|_{\ast\ast}\\
\leq&C\|(|\phi_{1}|+|\phi_{2}|)^{2_s^*-2}|\phi_{1}-\phi_{2}|\|_{\ast\ast}\\
\leq&\frac{1}{2}\|\phi_{1}-\phi_{2}\|_{\ast}.
\end{aligned}
\end{displaymath}
The case $2_s^*>3$ can be discussed in a similar way.

\vskip8pt

Hence,  $A$ is a contraction map. The Banach fixed point theorem
tells us that there exists a unique solution $\phi\in E$ for the
problem \eqref{problem4}.\par

\vskip8pt

Finally, by Proposition \ref{proposition1}, we have
$$\|\phi\|_{\ast}\leq C(\frac{1}{\lambda})^{s+\sigma}  \ \hbox{and} \ |c_l|\leq C\|\mathcal {F}(\phi)+l_k\|_{\ast\ast}\leq C(\frac{1}{\lambda})^{s+\sigma}.$$
\hfill{$\Box$}

\section{Proof of the Main Theorem }

Let $\phi$ be the function obtained in Proposition \ref{proposition2} and $u_k=Z_{\overline{r}, \overline{y}'', \lambda}+\phi$. In order to use local Pohozaev identities, we quote the extension of $u_k$, that is $\tilde{u}_{k}=\tilde{Z}_{\overline{r}, \overline{y}'', \lambda}+\tilde{\phi}$.
$\tilde{Z}_{\overline{r}, \overline{y}''}$ and $\tilde{\phi}$ are extensions of $Z_{\overline{r}, \overline{y}''}$ and $\phi$, respectively.
Then we have
\begin{equation}\label{widetildeuLm}
\begin{cases}
\displaystyle\mathrm{div}(t^{1-2s}\nabla \widetilde{u}_k)=0, \quad &\mbox{in} \quad \mathbb{R}^{N+1}_+, \\
\displaystyle-\lim\limits_{t\rightarrow 0^+}t^{1-2s}\partial_t \widetilde{u}_k=\omega_s \big(-V(r,y'')u_k+(u_k)_+^{2^*_s-1}+\sum\limits_{l=1}^Nc_l\sum\limits_{j=1}^kZ_{x_j,\lambda}^{2^*_s-2}Z_{j,l}\big), \quad  & \mbox{on } \quad \mathbb{R}^{N}.
\end{cases}
\end{equation}
Without loss of generality, we may assume $\omega_s=1$. Multiplying \eqref{widetildeuLm} by $\frac{\partial \tilde{u}_{k}}{\partial y_{i}}$ ($i=3,\ldots, N$) and $\langle\nabla \tilde{u}_k, Y\rangle$ respectively, integrating by parts, we have the following two Pohozaev identities:
\begin{eqnarray}\label{firstpohozaevidentity1}
\begin{aligned}
&\quad-\int_{\partial''\mathcal B^{+}_{\rho}}t^{1-2s}\frac{\partial \tilde{u}_{k}}{\partial \nu}\frac{\partial \tilde{u}_{k}}{\partial y_{i}}+\frac{1}{2}\int_{\partial''\mathcal B^{+}_{\rho}}t^{1-2s}|\nabla \tilde{u}_{k}|^{2}\nu_{i}\\
&=\int_{B_{\rho}}\big(-V(r,y'')u_k+(u_k)_+^{2^*_s-1}+\sum\limits_{l=1}^Nc_l\sum\limits_{j=1}^kZ_{x_j,\lambda}^{2^*_s-2}Z_{j,l}\big)\frac{\partial u_k}{\partial y_i}, \ \ \ \ \ \ i=3,\ldots, N,
\end{aligned}
\end{eqnarray}
and
\begin{eqnarray}\label{secondtpohozaevidentity1}
\begin{aligned}
&\quad-\int_{\partial''\mathcal B^{+}_{\rho}}t^{1-2s}\langle\nabla \tilde{u}_k, Y\rangle \frac{\partial\tilde{u}_k}{\partial\nu} +\frac{1}{2}\int_{\partial''\mathcal B^{+}_{\rho}}t^{1-2s}|\nabla \tilde{u}_k|^{2}\langle Y,\nu\rangle +\frac{2s-N}{2}\int_{\partial\mathcal B^{+}_{\rho}}t^{1-2s}\frac{\partial\tilde{u}_k}{\partial\nu} \tilde{u}_k\\
&=\int_{ B_{\rho}}\big(-V(r,y'')u_k+(u_k)_+^{2^*_s-1}+\sum\limits_{l=1}^Nc_l\sum\limits_{j=1}^kZ_{x_j,\lambda}^{2^*_s-2}Z_{j,l}\big)\langle y, u_k\rangle.
\end{aligned}
\end{eqnarray}

In the following, we assume $\rho\in(2\delta,5\delta)$. We have the following lemma.
\begin{lemma}\label{}
Suppose that $(\overline{r}, \overline{y}'', \lambda)$ satisfies
\begin{eqnarray}\label{firstpohozaevidentity}
\begin{aligned}
&\quad-\int_{\partial''\mathcal B^{+}_{\rho}}t^{1-2s}\frac{\partial \tilde{u}_{k}}{\partial \nu}\frac{\partial \tilde{u}_{k}}{\partial y_{i}}+\frac{1}{2}\int_{\partial''\mathcal B^{+}_{\rho}}t^{1-2s}|\nabla \tilde{u}_{k}|^{2}\nu_{i}\\
&=\int_{B_{\rho}}\big(-V(r,y'')u_k+(u_k)_+^{2^*_s-1}\big)\frac{\partial u_k}{\partial y_i}, \ \ \ \ \ \ i=3,\ldots, N,
\end{aligned}
\end{eqnarray}
\begin{eqnarray}\label{secondtpohozaevidentity}
\begin{aligned}
&\quad-\int_{\partial''\mathcal B^{+}_{\rho}}t^{1-2s}\langle\nabla \tilde{u}_k, Y\rangle \frac{\partial\tilde{u}_k}{\partial\nu} +\frac{1}{2}\int_{\partial''\mathcal B^{+}_{\rho}}t^{1-2s}|\nabla \tilde{u}_k|^{2}\langle Y,\nu\rangle +\frac{2s-N}{2}\int_{\partial\mathcal B^{+}_{\rho}}t^{1-2s}\frac{\partial\tilde{u}_k}{\partial\nu} \tilde{u}_k\\
&=\int_{ B_{\rho}}\big(-V(r,y'')u_k+(u_k)_+^{2^*_s-1}\big)\langle y, u_k\rangle
\end{aligned}
\end{eqnarray}
and
\begin{eqnarray}\label{thirdpohozaevidentity}
\begin{aligned}
\int_{\mathbb{R}^N}\big((-\Delta)^su_k+V(r,y'')u_k-(u_k)_+^{2^*_s-1}\big)\frac{\partial Z_{\overline{r}, \overline{y}'', \lambda}}{\partial\lambda}=0.
\end{aligned}
\end{eqnarray}
Then we have $c_l=0$, $l=1,\ldots, N$.
\end{lemma}
{\bf Proof.}  By \eqref{firstpohozaevidentity1}, \eqref{secondtpohozaevidentity1}, \eqref{firstpohozaevidentity} and \eqref{secondtpohozaevidentity}, we have
\begin{eqnarray}\label{cl1}
\begin{aligned}
\sum\limits_{l=1}^Nc_l\sum\limits_{j=1}^k\int_{B_{\rho}}Z_{x_j,\lambda}^{2^*_s-2}Z_{j,l}\frac{\partial u_k}{\partial y_i}=0, \ i=3,\ldots,N, \\  \sum\limits_{l=1}^Nc_l\sum\limits_{j=1}^k\int_{B_{\rho}}Z_{x_j,\lambda}^{2^*_s-2}Z_{j,l}\langle y, \nabla u_k\rangle=0.
\end{aligned}
\end{eqnarray}

Note that $\zeta=0$ in $\mathbb{R}^N\setminus B_{\rho}$. By \eqref{thirdpohozaevidentity} and \eqref{cl1}, we have
\begin{eqnarray}\label{aum1}
\begin{aligned}
\sum\limits_{l=1}^Nc_l\sum\limits_{j=1}^k\int_{\mathbb{R}^N}Z_{x_j,\lambda}^{2^*_s-2}Z_{j,l}v=\sum\limits_{l=1}^Nc_l\sum\limits_{j=1}^k\int_{B_{\rho}}Z_{x_j,\lambda}^{2^*_s-2}Z_{j,l}v=0,
\end{aligned}
\end{eqnarray}
for $v=\frac{\partial u_k}{\partial y_i}$, $v=\langle\nabla u_k, y\rangle $ and $v=\frac{\partial Z_{\overline{r}, \overline{y}'', \lambda}}{\partial\lambda}$.

By direct calculations, we have
\begin{eqnarray}\label{aum5}
\begin{aligned}
\sum\limits_{j=1}^k\int_{B_{\rho}}Z_{x_j,\lambda}^{2^*_s-2}Z_{j,2}\langle y', \nabla_{y'} Z_{\overline{r},\overline{y}'',\lambda}\rangle=k\lambda^{2}(a_1+o(1)),
\end{aligned}
\end{eqnarray}
\begin{eqnarray}\label{aum6}
\begin{aligned}
\sum\limits_{j=1}^k\int_{B_{\rho}}Z_{x_j,\lambda}^{2^*_s-2}Z_{j,i}\frac{\partial Z_{\overline{r},\overline{y}'',\lambda}}{\partial y_i}=k\lambda^{2}(a_2+o(1)), \ \ i=3,\ldots,N,
\end{aligned}
\end{eqnarray}
and
\begin{eqnarray}\label{aum7}
\begin{aligned}
\sum\limits_{j=1}^k\int_{B_{\rho}}Z_{x_j,\lambda}^{2^*_s-2}Z_{j,1}\frac{\partial Z_{\overline{r},\overline{y}'',\lambda}}{\partial \lambda}=\frac{k}{\lambda^{2}}(a_3+o(1)),
\end{aligned}
\end{eqnarray}
where $a_1>0$, $a_2>0$ and $a_3>0$.

Furthermore, we have that
\begin{eqnarray}\label{aum3}
\begin{aligned}
&\quad\sum\limits_{l=1}^Nc_l\sum\limits_{j=1}^k\int_{B_{\rho}}Z_{x_j,\lambda}^{2^*_s-2}Z_{j,l}\langle y, \nabla Z_{\overline{r},\overline{y}'',\lambda}\rangle\\
&=\sum\limits_{j=1}^k\int_{B_{\rho}}Z_{x_j,\lambda}^{2^*_s-2}Z_{j,2}\langle y', \nabla_{y'} Z_{\overline{r},\overline{y}'',\lambda}\rangle c_2+O(\frac{k}{\lambda^{N-2}}|c_2|)+o(k\lambda^2\sum\limits_{l=3}^N|c_l|)+o(k|c_1|)\\
&=k\lambda^2(a_1+o(1))c_2+o(k\lambda^2\sum\limits_{l=3}^N|c_l|)+o(k|c_1|)
\end{aligned}
\end{eqnarray}
and
\begin{eqnarray}\label{aum4}
\begin{aligned}
&\quad\sum\limits_{l=1}^Nc_l\sum\limits_{j=1}^k\int_{B_{\rho}}Z_{x_j,\lambda}^{2^*_s-2}Z_{j,l}\frac{\partial Z_{\overline{r},\overline{y}'',\lambda}}{\partial y_i}\\
&=\sum\limits_{j=1}^k\int_{B_{\rho}}Z_{x_j,\lambda}^{2^*_s-2}Z_{j,i}\frac{\partial Z_{\overline{r},\overline{y}'',\lambda}}{\partial y_i} c_i+o(k\lambda^2\sum\limits_{l\neq 1,i}|c_l|)+o(k|c_1|)\\
&=k\lambda^2(a_2+o(1))c_i+o(k\lambda^2\sum\limits_{l\neq 1,i}|c_l|)+o(k|c_1|), \ \ \ \ i=3,\ldots,N.
\end{aligned}
\end{eqnarray}

Since $\phi$ is a solution to \eqref{problem4}, by fractional elliptical equation estimates (see for example \cite{Silvestre}), we can obtain $\phi\in C^1$ when $s>\frac{1}{2}$.
Using integrating by parts and $\|\phi\|_*\leq\frac{C}{\lambda^{s+\sigma}}$, we have
\begin{eqnarray}\label{}
\begin{aligned}
\sum\limits_{l=1}^Nc_l\sum\limits_{j=1}^k\int_{B_{\rho}}Z_{x_j,\lambda}^{2^*_s-2}Z_{j,l}v=o(k\lambda^2\sum\limits_{l=2}^N|c_l|)+o(k|c_1|),
\end{aligned}
\end{eqnarray}
for $v=\langle y, \nabla\phi_{\overline{r},\overline{y}'',\lambda}\rangle$ and
$v=\frac{\partial\phi_{\overline{r},\overline{y}'',\lambda}}{\partial y_i}$.

\vskip8pt

It follows from \eqref{aum1}, we obtain that
\begin{eqnarray}\label{aum2}
\begin{aligned}
\sum\limits_{l=1}^Nc_l\sum\limits_{j=1}^k\int_{B_{\rho}}Z_{x_j,\lambda}^{2^*_s-2}Z_{j,l}v=o(k\lambda^2\sum\limits_{l=2}^N|c_l|)+o(k|c_1|),
\end{aligned}
\end{eqnarray}
for $v=\langle\nabla y, Z_{\overline{r},\overline{y}'',\lambda}\rangle$ and $v=\frac{\partial Z_{\overline{r},\overline{y}'',\lambda}}{\partial y_i}$.

By \eqref{aum3}, \eqref{aum4} and \eqref{aum2}, we have
\begin{eqnarray}\label{aum8}
\begin{aligned}
c_l=o(\frac{1}{\lambda^2}|c_1|), \ l=2,\ldots,N.
\end{aligned}
\end{eqnarray}

From \eqref{aum1}, \eqref{aum7} and \eqref{aum8}, we deduce that
\begin{eqnarray}\label{}
\begin{aligned}
0&=\sum\limits_{l=1}^Nc_l\sum\limits_{j=1}^k\int_{B_{\rho}}Z_{x_j,\lambda}^{2^*_s-2}Z_{j,1}\frac{\partial Z_{\overline{r},\overline{y}'',\lambda}}{\partial \lambda}\\
&=\sum\limits_{j=1}^k\int_{B_{\rho}}Z_{x_j,\lambda}^{2^*_s-2}Z_{j,1}\frac{\partial Z_{\overline{r},\overline{y}'',\lambda}}{\partial \lambda}c_1+o(\frac{k}{\lambda^2})c_1\\
&=k(a_3+o(1))c_1+o(\frac{k}{\lambda^2})c_1,
\end{aligned}
\end{eqnarray}
which implies that $c_1=0$. We also have $c_l=0, \ l=2,\ldots,N.$  \hfill{$\Box$}

Note that
\begin{eqnarray*}
\begin{aligned}
&\frac{2s-N}{2}\int_{\partial\mathcal B^{+}_{\rho}}t^{1-2s}\frac{\partial\tilde{u}_k}{\partial\nu} \tilde{u}_k\\
=&\frac{2s-N}{2}\int_{\partial''\mathcal B^{+}_{\rho}}t^{1-2s}\frac{\partial\tilde{u}_k}{\partial\nu} \tilde{u}_k+\frac{2s-N}{2}\int_{B_{\rho}}
\big(-V(r,y'')u_k+(u_k)_+^{2^*_s-1}+\sum\limits_{l=1}^Nc_l\sum\limits_{j=1}^kZ_{x_j,\lambda}^{2^*_s-2}Z_{j,l}\big)u_k,
\end{aligned}
\end{eqnarray*}
\begin{eqnarray*}
\begin{aligned}
&\quad\int_{B_{\rho}}\big(-V(r,y'')u_k+(u_k)_+^{2^*_s-1}\big)\langle y, \nabla u_k\rangle\\
&=\int_{B_{\rho}}\big(-\frac{1}{2}V(r,y'')\langle y, \nabla u_k^2\rangle+\frac{1}{2^*_s}\langle y, \nabla (u_k)_+^{2^*_s}\rangle\big)\\
&=-\frac{1}{2}\int_{\partial B_{\rho}}V(r,y'')u_k^2\langle y, \nu\rangle+\frac{1}{2}\int_{B_{\rho}}\big(NV(r,y'')+\langle\nabla V(r,y''), y\rangle\big)u_k^2\\
&\quad\quad+
\frac{1}{2^*_s}\int_{\partial B_{\rho}}(u_k)_+^{2^*_s}\langle y, \nu\rangle+\frac{2s-N}{2}\int_{B_{\rho}}(u_k)_+^{2^*_s},
\end{aligned}
\end{eqnarray*}
and $\sum\limits_{l=1}^Nc_l\int_{B_{\rho}}
\sum\limits_{j=1}^kZ_{x_j,\lambda}^{2^*_s-2}Z_{j,l}\phi=0$.
We find that \eqref{secondtpohozaevidentity} is equivalent to
\begin{eqnarray}\label{aum12}
\begin{aligned}
&\quad
\int_{ B_{\rho}}\big(sV(r,y'')+\frac{1}{2}\langle\nabla V(r,y''), y\rangle\big)u_k^2\\
&=-\int_{\partial''\mathcal B^{+}_{\rho}}t^{1-2s}\langle\nabla \tilde{u}_k, Y\rangle \frac{\partial\tilde{u}_k}{\partial\nu} +\frac{1}{2}\int_{\partial''\mathcal B^{+}_{\rho}}t^{1-2s}|\nabla \tilde{u}_k|^{2}\langle Y,\nu\rangle +\frac{2s-N}{2}\int_{\partial''\mathcal B^{+}_{\rho}}t^{1-2s}\frac{\partial\tilde{u}_k}{\partial\nu} \tilde{u}_k\\
&\quad+\frac{1}{2}\int_{\partial B_{\rho}}V(r,y'')u_k^2\langle y,\nu\rangle-\frac{1}{2^*_{s}}\int_{\partial B_{\rho}}(u_k)_+^{2^*_{s}}\langle y,\nu\rangle+\frac{2s-N}{2}\sum\limits_{l=1}^Nc_l\int_{B_{\rho}}
\sum\limits_{j=1}^kZ_{x_j,\lambda}^{2^*_s-2}Z_{j,l}Z_{\overline{r},\overline{y}'',\lambda}.
\end{aligned}
\end{eqnarray}

Similarly, \eqref{firstpohozaevidentity} is equivalent to
\begin{eqnarray}\label{aum11}
\begin{aligned}
&\quad\frac{1}{2}\int_{B_{\rho}}\frac{\partial V(r,y'')}{\partial y_i''}u_k^2\\
&=\int_{\partial''\mathcal B^{+}_{\rho}}t^{1-2s}\frac{\partial \tilde{u}_{k}}{\partial \nu}\frac{\partial \tilde{u}_{k}}{\partial y_{i}}-\frac{1}{2}\int_{\partial''\mathcal B^{+}_{\rho}}t^{1-2s}|\nabla \tilde{u}_{k}|^{2}\nu_{i}\\
&\quad+\frac{1}{2}\int_{\partial B_{\rho}}V(r,y'')u_k^2\nu_i
+\frac{1}{2_s^*}\int_{\partial B_{\rho}}u_k^{2^*_s}\nu_i  ,  \ \ i=3,\ldots, N.
\end{aligned}
\end{eqnarray}

\begin{lemma}\label{relation estimate}
Relations \eqref{aum12} and \eqref{aum11} are equivalent to
\begin{equation}\label{aum9}
\begin{aligned}
\int_{B_{\rho}}\big(sV(r,y'')+\frac{1}{2}\langle\nabla V(r,y''),y\rangle\big)u_k^2=O(\frac{k}{\lambda^{2s+\sigma}})
\end{aligned}
\end{equation}
and
\begin{equation}\label{aum10}
\begin{aligned}
\int_{B_{\rho}}\frac{\partial V(r,y'')}{\partial y_i}u_k^2=O(\frac{k}{\lambda^{2s+\sigma}}), \ \ i=3,\ldots,N.
\end{aligned}
\end{equation}

\end{lemma}

\begin{proof}  We only give the proof for \eqref{aum9}. The proof of \eqref{aum10} is similar.

Note that $\tilde{u}_k=\tilde{Z}_{\overline{r},\overline{y}'',\lambda}+\tilde{\phi}$. We have
\begin{eqnarray}\label{}
\begin{aligned}
&\quad\int_{\partial''\mathcal B^{+}_{\rho}}t^{1-2s}\langle\nabla \tilde{u}_k, Y\rangle \frac{\partial\tilde{u}_k}{\partial\nu}\\
&=\int_{\partial''\mathcal B^{+}_{\rho}}t^{1-2s}\langle\nabla \widetilde{Z}_{\overline{r},\overline{y}'',\lambda}, Y\rangle \frac{\partial\widetilde{Z}_{\overline{r},\overline{y}'',\lambda}}{\partial\nu}+\int_{\partial''\mathcal B^{+}_{\rho}}t^{1-2s}\langle\nabla \widetilde{\phi}, Y\rangle \frac{\partial\widetilde{\phi}}{\partial\nu}\\
&\quad\quad+\int_{\partial''\mathcal B^{+}_{\rho}}t^{1-2s}\langle\nabla \widetilde{Z}_{\overline{r},\overline{y}'',\lambda}, Y\rangle \frac{\partial\widetilde{\phi}}{\partial\nu}
+\int_{\partial''\mathcal B^{+}_{\rho}}t^{1-2s}\langle\nabla \widetilde{\phi}, Y\rangle \frac{\partial\widetilde{Z}_{\overline{r},\overline{y}'',\lambda}}{\partial\nu}.
\end{aligned}
\end{eqnarray}

Using Lemma \ref{estimate for Z}, we can obtain
\begin{eqnarray}\label{aum13}
\begin{aligned}
&\quad|\int_{\partial''\mathcal B^{+}_{\rho}}t^{1-2s}\langle\nabla \widetilde{Z}_{\overline{r},\overline{y}'',\lambda}, Y\rangle \frac{\partial\widetilde{Z}_{\overline{r},\overline{y}'',\lambda}}{\partial\nu}|\\
&\leq \frac{C}{\lambda^{N-2s}}\int_{\partial''\mathcal B^{+}_{\rho}}t^{1-2s}(\sum\limits_{i=1}^k\frac{1}{(1+|y-x_{i}|)^{N-2s+1}})^2\\
&\leq \frac{Ck^2}{\lambda^{N-2s}}\int_{\partial''\mathcal B^{+}_{\rho}}\frac{t^{1-2s}}{(1+|y-x_{1}|)^{2N-4s+2}}\\
&\leq \frac{Ck^2}{\lambda^{N-2s}}.
\end{aligned}
\end{eqnarray}
By Lemma \ref{le:lemma appendixB.4},
\begin{equation}\label{aum15}
\begin{aligned}
|\int_{\partial''\mathcal B^{+}_{\rho}}t^{1-2s}\langle\nabla \widetilde{\phi}, Y\rangle \frac{\partial\widetilde{\phi}}{\partial\nu}|\leq C\int_{\partial''\mathcal B^{+}_{\rho}}t^{1-2s}|\nabla\widetilde{\phi}|^2\leq \frac{Ck\|\phi\|_*^2}{\lambda^{\tau}}.
\end{aligned}
\end{equation}
By the process of the proof of  \eqref{aum13} and \eqref{aum15}, we also have
\begin{equation*}
\begin{aligned}
|\int_{\partial''\mathcal B^{+}_{\rho}}t^{1-2s}\langle\nabla \widetilde{Z}_{\overline{r},\overline{y}'',\lambda}, Y\rangle \frac{\partial\widetilde{\phi}}{\partial\nu}
+\int_{\partial''\mathcal B^{+}_{\rho}}t^{1-2s}\langle\nabla \widetilde{\phi}, Y\rangle \frac{\partial\widetilde{Z}_{\overline{r},\overline{y}'',\lambda}}{\partial\nu}|\leq\frac{Ck\|\phi\|_*}{\lambda^{\frac{N-2s}{2}}}.
\end{aligned}
\end{equation*}
Note that $N>4s+2\tau$. So we have proved that $$|\int_{\partial''\mathcal B^{+}_{\rho}}t^{1-2s}\langle\nabla \tilde{u}_k, Y\rangle \frac{\partial\tilde{u}_k}{\partial\nu}|\leq \frac{Ck}{\lambda^{2s+\sigma}}.$$

Similarly, we can prove $$|\int_{\partial''\mathcal B^{+}_{\rho}}t^{1-2s}|\nabla \tilde{u}_k|^{2}\langle Y,\nu\rangle|\leq \frac{Ck}{\lambda^{2s+\sigma}}.$$

Next, we estimate the term $\int_{\partial''\mathcal B^{+}_{\rho}}t^{1-2s}\frac{\partial\tilde{u}_k}{\partial\nu} \tilde{u}_k$,
\begin{equation*}
\begin{aligned}
&\quad\int_{\partial''\mathcal B^{+}_{\rho}}t^{1-2s}\frac{\partial\tilde{u}_k}{\partial\nu} \tilde{u}_k\\
&=\int_{\partial''\mathcal B^{+}_{\rho}}t^{1-2s}\frac{\partial\widetilde{Z}_{r'',\overline{y}'',\lambda}}{\partial\nu} \widetilde{Z}_{r'',\overline{y}'',\lambda}+\int_{\partial''\mathcal B^{+}_{\rho}}t^{1-2s}\frac{\partial\widetilde{\phi}}{\partial\nu}\widetilde{\phi}\\
&\quad+\int_{\partial''\mathcal B^{+}_{\rho}}t^{1-2s}\frac{\partial\widetilde{Z}_{r'',\overline{y}'',\lambda}}{\partial\nu} \widetilde{\phi}+\int_{\partial''\mathcal B^{+}_{\rho}}t^{1-2s}\frac{\partial\widetilde{\phi}}{\partial\nu} \widetilde{Z}_{r'',\overline{y}'',\lambda}.
\end{aligned}
\end{equation*}
By Lemma \ref{estimate for Z},
\begin{eqnarray}
\begin{aligned}
&\quad|\int_{\partial''\mathcal B^{+}_{\rho}}t^{1-2s}\frac{\partial\widetilde{Z}_{r'',\overline{y}'',\lambda}}{\partial\nu} \widetilde{Z}_{r'',\overline{y}'',\lambda}|\\
&\leq \frac{C}{\lambda^{N-2s}}\int_{\partial''\mathcal B^{+}_{\rho}}t^{1-2s}\sum\limits_{i=1}^k\frac{1}{(1+|y-x_{i}|)^{N-2s+1}}\times\sum\limits_{j=1}^k\frac{1}{(1+|y-x_{j}|)^{N-2s}}\\
&\leq\frac{Ck^2}{\lambda^{N-2s}}\int_{\partial''\mathcal B^{+}_{\rho}}\frac{t^{1-2s}}{(1+|y-x_{1}|)^{2N-4s+1}}\\
&\leq \frac{Ck^2}{\lambda^{N-2s}}.
\end{aligned}
\end{eqnarray}
It follows from \eqref{B.4.2} that
\begin{equation}\label{aum14}
\begin{aligned}
\int_{\partial''\mathcal B^{+}_{\rho}}t^{1-2s}|\widetilde{\phi}|^2\leq \frac{Ck\|\phi\|_*^2}{\lambda^{\tau}}.
\end{aligned}
\end{equation}
By \eqref{aum14} and Lemma \ref{le:lemma appendixB.4}, one has
\begin{equation*}
\begin{aligned}
&\quad|\int_{\partial''\mathcal B^{+}_{\rho}}t^{1-2s}\frac{\partial\widetilde{\phi}}{\partial\nu}\widetilde{\phi}dS|\\
&\leq(\int_{\partial''\mathcal B^{+}_{\rho}}t^{1-2s}|\nabla\widetilde{\phi}|^2)^{\frac{1}{2}}(\int_{\partial''\mathcal B^{+}_{\rho}}t^{1-2s}\widetilde{\phi}^2)^{\frac{1}{2}}\\
&\leq \frac{Ck\|\phi\|_*^2}{\lambda^{\tau}}.
\end{aligned}
\end{equation*}
Similarly, we can get $$|\int_{\partial''\mathcal B^{+}_{\rho}}t^{1-2s}\frac{\partial\widetilde{Z}_{r'',\overline{y}'',\lambda}}{\partial\nu} \widetilde{\phi}+\int_{\partial''\mathcal B^{+}_{\rho}}t^{1-2s}\frac{\partial\widetilde{\phi}}{\partial\nu} \widetilde{Z}_{r'',\overline{y}'',\lambda}|\leq \frac{Ck\|\phi\|_*^2}{\lambda^{\tau}}.$$

We have proved that $$|\int_{\partial''\mathcal B^{+}_{\rho}}t^{1-2s}\frac{\partial\tilde{u}_k}{\partial\nu} \tilde{u}_k|\leq\frac{Ck}{\lambda^{2s+\sigma}}.$$

Since $\zeta=0$ on $\partial B_{\rho}$, $u_k=\phi$ on $\partial B_{\rho}$. We deduce that
\begin{equation*}
\begin{aligned}
|\int_{\partial B_{\rho}}V(r,y'')u_k^2\langle y,\nu\rangle|
&\leq C\|\phi\|_*^2\int_{\partial B_{\rho}}\big(\sum\limits_{j=1}^k\frac{\lambda^{\frac{N-2s}{2}}}{(1+\lambda|y-x_j|)^{\frac{N-2s}{2}+\tau}}\big)^2\\
&\leq \frac{Ck^2\|\phi\|_*^2}{\lambda^{2\tau}}\leq \frac{Ck}{\lambda^{2s+\tau}}
\end{aligned}
\end{equation*}
and
\begin{equation*}
\begin{aligned}
|\int_{\partial B_{\rho}}(u_k)_+^{2^*_{s}}\langle y,\nu\rangle|\leq\frac{Ck^{2^*_s}\|\phi\|_*^{2^*_s}}{\lambda^{2^*_s\tau}}\leq\frac{Ck}{\lambda^{2s+\tau}}.
\end{aligned}
\end{equation*}

From Proposition \ref{proposition2}, we know the following estimates for $c_l$
\begin{equation}
\begin{aligned}
|c_l|\leq C(\frac{1}{\lambda})^{s+\sigma}.
\end{aligned}
\end{equation}
On the other hand,
\begin{equation}
\begin{aligned}
&\sum\limits_{j=1}^k\int_{B_{\rho}}
Z_{x_j,\lambda}^{2^*_s-2}Z_{j,l}Z_{\overline{r},\overline{y}'',\lambda}\\
=&\sum\limits_{j=1}^k\int_{B_{\rho}}
Z_{x_j,\lambda}^{2^*_s-1}Z_{j,l}+\sum\limits_{j=1}^k\int_{B_{\rho}}
\sum\limits_{i\neq j}Z_{x_j,\lambda}^{2^*_s-2}Z_{j,l}Z_{x_i,\lambda}\\
=&O(\frac{1}{\lambda^N})+O(\frac{k}{\lambda^{2s}}).
\end{aligned}
\end{equation}
These imply that
\begin{equation}
\begin{aligned}
&|\sum\limits_{l=1}c_l\sum\limits_{j=1}^k\int_{B_{\rho}}
Z_{x_j,\lambda}^{2^*_s-2}Z_{j,l}Z_{\overline{r},\overline{y}'',\lambda}|\leq\frac{Ck}{\lambda^{2s+\sigma}}.
\end{aligned}
\end{equation}

Combining the above estimates, we find that \eqref{aum12} is equivalent to
\begin{equation*}
\begin{aligned}
\int_{B_{\rho}}\big(sV(r,y'')+\frac{1}{2}\langle\nabla V(r,y''),y\rangle\big)u_k^2=O(\frac{k}{\lambda^{2s+\sigma}}).
\end{aligned}
\end{equation*}

\end{proof}

\begin{lemma}\label{gry}
For any function $g(r,y'') \in C^1(\mathbb{R}^N)$, it holds
\[
\ds \int_{B_\rho}g(r,y'')u_k^2=k\left(\frac{1}{\la^{2s}}g(\bar{r},\bar{y}'')\ds \int_{ \mathbb{R}^N}U_{0,1}^2+o\big(\frac{1}{\la^{2s}}\big)\right).
\]
\end{lemma}
\begin{proof}
We have
\[
\ds \int_{B_\rho}g(r,y'')u_k^2=\ds \int_{ D_\rho}g(r,y'')Z_{\bar{r},\bar{y}'',\la}^2+2\ds \int_{ D_\rho}g(r,y'')Z_{\bar{r},\bar{y}'',\la}\phi+\ds \int_{ D_\rho}g(r,y'')\phi^2.
\]

Note that
\begin{equation*}
\begin{aligned}
&\quad\big|2\int_{B_\rho}g(r,y'')Z_{\bar{r},\bar{y}'',\la}\phi+\int_{B_\rho}g(r,y'')\phi^2\big|\\
&\leq C\big(\|\phi\|_*\int_{B_\rho}\sum\limits_{i=1}^k\frac{\zeta\lambda^{N-2s}}{(1+\lambda|y-x_i|)^{N-2s}}\sum\limits_{j=1}^k\frac{1}{(1+\lambda|y-x_j|)^{\frac{N-2s}{2}+\tau}}\\
&\quad\quad\quad+\|\phi\|_*^2\int_{B_\rho}\big(\sum\limits_{i=1}^k\frac{\lambda^{\frac{N-2s}{2}}}{(1+\lambda|y-x_i|)^{\frac{N-2s}{2}+\tau}}\big)^2\big)\\
&\leq\frac{Ck\|\phi\|_*}{\lambda^{s}}+\frac{Ck\|\phi\|_*^2}{\lambda^{2\tau}}\leq\frac{Ck}{\lambda^{2s+\sigma}}
\end{aligned}
\end{equation*}
and
\begin{equation*}
\begin{aligned}
\int_{B_\rho}g(r,y'')Z_{\bar{r},\bar{y}'',\la}^2&=\sum\limits_{j=1}^k\left(\int_{B_\rho}g(r,y'')Z^2_{x_j,\la}+\sum\limits_{i\neq j}\int_{B_\rho}g(r,y'')Z_{x_i,\la}Z_{x_j,\la}\right)\\
&=k\left(\frac{1}{\lambda^{2s}}g(\bar{r},\bar{y}'')\int_{\mathbb{R}^N}U^2_{0,1}+o(\frac{1}{\lambda^{2s}})\right).
\end{aligned}
\end{equation*}
We get the result.

\end{proof}

{\bf Proof of Theorem \ref{th:1}.}  By \eqref{aum9} and \eqref{aum10}, we can deduce that
\begin{equation*}
\begin{aligned}
\int_{B_{\rho}}\big(sV(r,y'')+\frac{1}{2}r\frac{\partial V(r,y'')}{\partial r}\big)u_k^2=O(\frac{k}{\lambda^{2s+\sigma}}).
\end{aligned}
\end{equation*}
That is
\begin{equation}\label{aum18}
\begin{aligned}
\int_{B_{\rho}}\frac{1}{r^{2s-1}}\frac{\partial (r^{2s}V(r,y''))}{\partial r}u_k^2=O(\frac{k}{\lambda^{2s+\sigma}}).
\end{aligned}
\end{equation}
Applying Lemma \ref{gry} to \eqref{aum10} and \eqref{aum18}, we obtain
\begin{equation*}
\begin{aligned}
k\big(\frac{1}{\lambda^{2s}}\frac{\partial V(\overline{r},\overline{y}'')}{\partial \overline{y}_i}\int_{\mathbb{R}^N}U_{0,1}^2+o(\frac{1}{\lambda^{2s}})\big)=o(\frac{k}{\lambda^{2s}})
\end{aligned}
\end{equation*}
and
\begin{equation*}
\begin{aligned}
k\big(\frac{1}{\lambda^{2s}}\frac{1}{\overline{r}^{2s-1}}\frac{\partial (\overline{r}^{2s}V(\overline{r},\overline{y}''))}{\partial \overline{r}}\int_{\mathbb{R}^N}U_{0,1}^2+o(\frac{1}{\lambda^{2s}})\big)=o(\frac{k}{\lambda^{2s}}).
\end{aligned}
\end{equation*}
Therefore, the equations to determine $(\overline{r},\overline{y}'')$ are
\begin{equation}\label{aum21}
\begin{aligned}
\frac{\partial (\overline{r}^{2s}V(\overline{r},\overline{y}''))}{\partial \overline{y}_i}=o(1), \ \ i=3,\ldots,N,
\end{aligned}
\end{equation}
and
\begin{equation}\label{aum22}
\begin{aligned}
\frac{\partial (\overline{r}^{2s}V(\overline{r},\overline{y}''))}{\partial \overline{r}}=o(1).
\end{aligned}
\end{equation}

From \eqref{thirdpohozaevidentity} and \eqref{energy expansion1}, the equation to determine $\lambda$ is
\begin{equation}\label{aum19}
\begin{aligned}
-\frac{B_1}{\la^{2s+1}}V(\bar{r},\bar{y}'')+ \frac{B_3k^{N-2s}}{\la^{N-2s+1}}=O\big(\frac{1}{\la^{2s+1+\sigma}}\big).
\end{aligned}
\end{equation}
Let $\lambda=tk^{\frac{N-2s}{N-4s}}$, then $t\in[L_0,L_1]$. It follows from \eqref{aum19} that
\begin{equation}\label{aum20}
\begin{aligned}
-\frac{B_1}{t^{2s+1}}V(\bar{r},\bar{y}'')+ \frac{B_3}{t^{N-2s+1}}=o(1), \ \  t\in[L_0,L_1].
\end{aligned}
\end{equation}

Define $$H(t,\bar{r},\bar{y}'')=\big(\nabla_{\overline{r},\overline{y}''}(\overline{r}^{2s}V(\overline{r},\overline{y}'')),
-\frac{B_1}{t^{2s+1}}V(\bar{r},\bar{y}'')+ \frac{B_3}{t^{N-2s+1}}\big).$$
Then
$$deg\big(H(t,\bar{r},\bar{y}''),[L_0,L_1]\times B_\theta((r_0,y''_0))\big)=-deg\big(\nabla_{\overline{r},\overline{y}''}(\overline{r}^{2s}V(\overline{r},\overline{y}'')), B_\theta((r_0,y''_0)\big)\neq0.$$
Hence, \eqref{aum21}, \eqref{aum21} and \eqref{aum20} have a solution $t_k\in[L_0,L_1]$ and $(\overline{r}_k,\overline{y}''_k)\in B_\theta((r_0,y''_0))$. \hfill{$\Box$}


\begin{appendices}

\section{Some estimates}

In this section, we give some essential estimates.

For $x_{i},x_{j},y\in {\Bbb R}^{N}$, define
$g_{ij}(y)=\frac{1}{(1+|y-x_{i}|)^{\alpha}(1+|y-x_{j}|)^{\beta}},$
where $x_{i}\neq x_{j},$  $\alpha>0$ and $\beta>0$ are two constants.
\begin{lemma}\label{Lemma appendix1}
For any constant $\gamma\in(0,\min(\alpha,\beta)]$, we have
$$g_{ij}(y)\leq\frac{C}{(1+|x_{i}-x_{j}|)^{\gamma}}\Big(\frac{1}{(1+|y-x_{i}|)^{\alpha+\beta-\gamma}}
+\frac{1}{(1+|y-x_{j}|)^{\alpha+\beta-\gamma}}\Big).$$
\end{lemma}
\begin{proof}
See the proof of Lemma A.1 in \cite{{wy}}.
\end{proof}

\begin{lemma}\label{Lemma appendix2}
For any constant $0<\vartheta< N-2s$, there is a constant $C>0$, such that
$$\int_{{\Bbb R}^{N}}\frac{1}{|y-z|^{N-2s}}\frac{1}{(1+|z|)^{2s+\vartheta}}dz\leq \frac{C}{(1+|y|)^{\vartheta}}.$$
\end{lemma}
\begin{proof}
See the proof of Lemma 2.1 in \cite{Guo2016}.
\end{proof}

\begin{lemma}\label{Lemma appendixaA30}
Let $\mu>0$. For any constants $0<\beta<N$, there exists a constant $C>0$ independent of $\mu$, such that
$$\int_{\mathbb{R}^N\setminus B_\mu(y)}\frac{1}{|y-z|^{N+2s}}\frac{1}{(1+|z|)^{\beta}}dz\leq C\big(\frac{1}{(1+|y|)^{\beta+2s}}+\frac{1}{\mu^{2s}}\frac{1}{(1+|y|)^{\beta}}\big).$$
\end{lemma}
\begin{proof} Without loss of generality, we set $|y|\geq2$, and
let $d=\frac{|y|}{2}.$
Then, we have
\begin{equation*}
\begin{aligned}
&\int_{\mathbb{R}^N\setminus B_\mu(y)}\frac{1}{|y-z|^{N+2s}}\frac{1}{(1+|z|)^{\beta}}dz\\
\leq& \int_{B_{d}(0)}+\int_{B_{d}(y)\setminus B_\mu(y)}+\int_{\mathbb{R}^{N}\setminus\big(B_{d}(0)\cup B_{d}(y)\big)}\frac{1}{|y-z|^{N+2s}}\frac{1}{(1+|z|)^{\beta}}dz.
\end{aligned}
\end{equation*}

By direct computation, we have
$$\int_{B_{d}(0)}\frac{dz}{|y-z|^{N+2s}(1+|z|)^{\beta}}\leq\frac{C}{d^{N+2s}}\int_{0}^{d}\frac
{r^{N-1}dr}{(1+r)^{\beta}}\leq\frac{C}{d^{\beta+2s}},$$
and
$$\int_{B_{d}(y)\setminus B_\mu(y)}\frac{dz}{|y-z|^{N+2s}(1+|z|)^{\beta}}\leq\frac{C}{d^{\beta}}\int_{B_{d}(y)\setminus B_\mu(y)}\frac{dz}
{|y-z|^{N+2s}}\leq\frac{C}{\mu^{2s}d^{\beta}}.$$

For $z\in\mathbb{R}^{N}\setminus\big(B_{d}(0)\cup B_{d}(y)\big)$, we
have $|y-z|\geq\frac{|y|}{2}$ and $|z|\geq\frac{|y|}{2}.$  If $|z|\geq 2|y|$, then $|y-z|\geq|z|-|y|\geq\frac{|z|}{2}$, and if $|z|<2|y|$, then $|y-z|\geq\frac{|y|}{2}>\frac{|z|}{4}$.
Thus, we have
\begin{equation*}
\begin{aligned}
\displaystyle\int_{\mathbb{R}^{N}\setminus\big(B_{d}(0)\cup
B_{d}(y)\big)}\frac{dz}{|y-z|^{N+2s}(1+|z|)^{\beta}} &\leq
C\displaystyle\int_{\mathbb{R}^{N}\setminus B_{d}(0)}\frac{dz}{(1+|z|)^{\beta}|z|^{N+2s}}\\
&\leq\frac{C}{d^{\beta+2s}}.
\end{aligned}
\end{equation*}

\end{proof}

\begin{lemma}\label{Lemma appendixa}
Let $\rho>0$. Suppose that $(y-x)^2+t^2=\rho^2$, $t>0$ and $\alpha>N$ . Then, when $0<\beta<N$, we have
\begin{equation}\label{lemmaA3 1}
\int_{\mathbb{R}^N}\frac{1}{(t+|z|)^{\alpha}} \frac{1}{|y-z-x|^{\beta}}dz\leq C\big(\frac{1}{(1+|y-x|)^{\beta}}\frac{1}{t^{\alpha-N}}+\frac{1}{(1+|y-x|)^{\alpha+\beta-N}}\big).
\end{equation}
\end{lemma}
\begin{proof}
The proof is same to that of Lemma A.3 in \cite{GNNT}.
\end{proof}

\begin{lemma}\label{estimate for Z}
Suppose that $(y-x)^2+t^2=\rho^2$. Then there exists a constant $C>0$ such that
\begin{equation}
\begin{aligned}
|\widetilde{Z}_{x_{i},\lambda}|\leq\frac{C}{\lambda^{\frac{N-2s}{2}}}\frac{1}{(1+|y-x_{i}|)^{N-2s}} \ \    \hbox{and}   \ \ |\nabla \tilde{Z}_{x_{i},\lambda}|\leq\frac{C}{\lambda^{\frac{N-2s}{2}}}\frac{1}{(1+|y-x_{i}|)^{N-2s+1}}.
\end{aligned}
\end{equation}
\end{lemma}

\begin{proof}
By Lemma \ref{Lemma appendixa}, we have
\begin{equation*}
\begin{aligned}
|\widetilde{Z}_{x_{i},\lambda}(y,t)|&=|\beta(N,s)\int_{\mathbb{R}^N}\frac{t^{2s}}{(|y-\xi|^2+t^2)^{\frac{N+2s}{2}}}\zeta(\xi)U_{x_{i},\lambda}(\xi)d\xi|\\
&=|\beta(N,s)C(N,s)\int_{\mathbb{R}^N}\frac{t^{2s}}{(|y-\xi|^2+t^2)^{\frac{N+2s}{2}}}\zeta(\xi)(\frac{\lambda}{1+\lambda^2|
\xi-x_{i}|^2})^{\frac{N-2s}{2}}d\xi|\\
&\leq \frac{C}{\lambda^{\frac{N-2s}{2}}}\int_{\mathbb{R}^N}\frac{1}{(1+|z|)^{N+2s}} \frac{1}{(\lambda^{-1}+|y-tz-x_{i}|)^{N-2s}}dz\\
&\leq \frac{C}{\lambda^{\frac{N-2s}{2}}}\int_{\mathbb{R}^N}\frac{t^{2s}}{(t+|z|)^{N+2s}} \frac{1}{(\lambda^{-1}+|y-z-x_{i}|)^{N-2s}}dz\\
&\leq\frac{C}{\lambda^{\frac{N-2s}{2}}}\frac{1}{(1+|y-x_{i}|)^{N-2s}}.
\end{aligned}
\end{equation*}

Note that for $l=1,\ldots,N$
\begin{equation*}
\begin{aligned}
&\quad \frac{\partial}{\partial y_l}\int_{\mathbb{R}^N}\frac{t^{2s}}{(|y-\xi|^2+t^2)^{\frac{N+2s}{2}}}\zeta(\xi)(\frac{\lambda}{1+\lambda^2|
\xi-x_{i}|^2})^{\frac{N-2s}{2}}d\xi\\
&=\frac{1}{\lambda^{\frac{N-2s}{2}}}\frac{\partial}{\partial y_l}\int_{\mathbb{R}^N}\frac{1}{(1+|z|^2)^{\frac{N+2s}{2}}}\zeta(y-tz)(\frac{1}
{\lambda^{-2}+|y-tz-x_{i}|^2})^{\frac{N-2s}{2}}dz\\
&=\frac{2s-N}{\lambda^{\frac{N-2s}{2}}}\int_{\mathbb{R}^N}\frac{1}{(1+|z|^2)^{\frac{N+2s}{2}}} \zeta(y-tz) \frac{(y-tz-x_{i})_l}{(\lambda^{-2}+|y-tz-x_{i}|^2)^{\frac{N-2s}{2}+1}}dz\\
&\quad+\frac{1}{\lambda^{\frac{N-2s}{2}}}\int_{\mathbb{R}^N}\frac{1}{(1+|z|^2)^{\frac{N+2s}{2}}}\frac{\partial\zeta(y-tz)}{\partial y_l} \frac{1}{(\lambda^{-2}+|y-tz-x_{i}|^2)^{\frac{N-2s}{2}}}dz,
\end{aligned}
\end{equation*}
and
\begin{equation*}
\begin{aligned}
&\quad \frac{\partial}{\partial t}\int_{\mathbb{R}^N}\frac{t^{2s}}{(|y-\xi|^2+t^2)^{\frac{N+2\sigma}{2}}}(\frac{\lambda}
{1+\lambda^2|\xi-x_{i}|^2})^{\frac{N-2s}{2}}d\xi\\
&=\frac{N-2s}{\lambda^{\frac{N-2s}{2}}}\int_{\mathbb{R}^N}\frac{1}{(1+|z|^2)^{\frac{N+2s}{2}}}\zeta(y-tz) \frac{\sum^{N}_{l=1}(y-tz-x_{k,L})_lz_l}{(\lambda^{-2}+|y-tz-x_{i}|^2)^{\frac{N-2s}{2}+1}}dz\\
&\quad+\frac{1}{\lambda^{\frac{N-2s}{2}}}\int_{\mathbb{R}^N}\frac{1}{(1+|z|^2)^{\frac{N+2s}{2}}} \frac{\nabla\zeta(y-tz)\cdot z}{(\lambda^{-2}+|y-tz-x_{i}|^2)^{\frac{N-2s}{2}}}dz.
\end{aligned}
\end{equation*}
Then, by the definition of $\zeta$ and \eqref{lemmaA3 1}, we have
\begin{equation}\label{nie1}
\begin{aligned}
|\nabla \tilde{Z}_{x_{i},\lambda}|&\leq \frac{C}{\lambda^{\frac{N-2s}{2}}}\int_{\mathbb{R}^N}\frac{1}{(1+|z|)^{N+2s-1}} \frac{1}{(1+|y-tz-x_{i}|)^{N-2s+1}}dz\\
&\leq \frac{C}{\lambda^{\frac{N-2s}{2}}}\int_{\mathbb{R}^N}\frac{t^{2s-1}}{(t+|z|)^{N+2s-1}} \frac{1}{(1+|y-z-x_{i}|)^{N-2s+1}}dz\\
&\leq \frac{C}{\lambda^{\frac{N-2s}{2}}}\frac{1}{(1+|y-x_{i}|)^{N-2s+1}}.
\end{aligned}
\end{equation}
\end{proof}

For any $\delta>0$, we define the following two sets
$$D_1=\{Y=(y,t):\delta<|Y-(r_0,y_0'',0)|<6\delta,t>0\},$$
and
$$D_2=\{Y=(y,t):2\delta<|Y-(r_0,y_0'',0)|<5\delta,t>0\}.$$
\begin{lemma}\label{le:lemma appendixB.4}
 For any $\delta>0$, there is a $\rho=\rho(\delta)\in(2\delta, 5\delta)$ such that
\begin{equation}\label{B.4.1}
\begin{aligned}
\int_{\partial''\mathcal B^{+}_{\rho}}t^{1-2s}|\nabla\widetilde{\phi}|^2dS\leq \frac{Ck\|\phi\|_*^2}{\lambda^{\tau}},
\end{aligned}
\end{equation}
where $C$ is a constant, dependent on $\delta$.
\end{lemma}

\begin{proof}
By \eqref{lemmaA3 1}, for $(y,t)\in D_1$, we have
\begin{equation}\label{B.4.2}
\begin{aligned}
|\widetilde{\phi}(y,t)|&=\Big|\int_{\mathbb{R}^N}\beta(N,s)\frac{t^{2s}}
{(|y-\xi|^2+t^2)^{\frac{N+2s}{2}}}\phi(\xi)d\xi\Big|\\
&\leq \frac{C\|\phi\|_{\ast}t^{2s}}{\lambda^\tau}
\sum^{k}_{i=1}\int_{\mathbb{R}^N}\frac{1}{(|z|+t)^{N+2s}}
\frac{1}{|y-z-x_{i}|^{\frac{N-2s}{2}+\tau}}dz\\
&\leq \frac{C\|\phi\|_{\ast}t^{2s}}{\lambda^\tau}
\sum^{k}_{i=1}(\frac{1}{(1+|y-x_{i}|)^{\frac{N-2s}{2}+\tau}}\frac{1}{t^{2s}}+\frac{1}{(1+|y-x_{i}|)^{\frac{N+2s}{2}+\tau}})\\
&\leq \frac{C\|\phi\|_{\ast}}{\lambda^\tau}
\sum^{k}_{i=1}\frac{1}{(1+|y-x_{i}|)^{\frac{N-2s}{2}+\tau}}.
\end{aligned}
\end{equation}
Take $\varphi\in C_0^{\infty}(\mathbb{R}^{N+1})$ be a function with $\varphi(y,t)=1$ in $D_2$, $\varphi(y,t)=0$ in $\mathbb{R}^{N+1}\setminus D_1$ and $|\nabla\varphi|\leq C$.
Note that $\widetilde{\phi}$ satisfies
\begin{equation}
\left\{\begin{aligned}
&\displaystyle-\mathrm{div}(t^{1-2s}\nabla\widetilde{\phi})=0\quad \mbox{in} \quad \mathbb{R}^{N+1}_+,\\
&\quad\displaystyle-\lim\limits_{t\to 0}t^{1-2s}\partial_t \widetilde{\phi}(y,t)\\&=-V(r,y'')\phi+(2_s^*-1)(Z_{\overline{r},\overline{y}'',\lambda})^{2_s^*-2}\phi+\mathcal {F}(\phi)+l_k+\sum\limits_{l=1}^Nc_l\sum\limits_{i=1}^kZ^{2_s^*-2}_{x_i,\lambda}Z_{i,l}, \quad \mbox{in}\quad {\Bbb R}^{N}.
\end{aligned}\right.
\end{equation}
Multiplying $\varphi^2\widetilde{\phi}$  on the  both sides of the equation and integrating by parts over $D_1$, we have
\begin{equation*}
\begin{aligned}
0=&\int_{D_1}-\mathrm{div}(t^{1-2s}\nabla\widetilde{\phi})\varphi^2\widetilde{\phi}dydt\\
=&\int_{D_1}t^{1-2s}\nabla\widetilde{\phi}\nabla(\varphi^2\widetilde{\phi})dydt\\
=&\int_{D_1}t^{1-2s}\nabla\widetilde{\phi}(\varphi^2\nabla\widetilde{\phi}+2\varphi\nabla\varphi\widetilde{\phi})dydt.
\end{aligned}
\end{equation*}
For any $\epsilon>0$, we have
\begin{equation*}
\begin{aligned}
&\quad\int_{D_1}t^{1-2s}\nabla\widetilde{\phi}\varphi\nabla\varphi\widetilde{\phi}dydt\\
&\leq \epsilon\int_{D_1}t^{1-2s}|\nabla\widetilde{\phi}|^2\varphi^2dydt     +C(\epsilon)\int_{D_1}t^{1-2s}\widetilde{\phi}^2|\nabla\varphi|^2dydt.
\end{aligned}
\end{equation*}
Taking $\epsilon=\frac{1}{4}$ and using \eqref{B.4.2}, we obtain that
\begin{equation*}
\begin{aligned}
\int_{D_2}t^{1-2s}|\nabla\widetilde{\phi}|^2dydt&\leq C\int_{D_1}t^{1-2s}\widetilde{\phi}^2|\nabla\varphi|^2\\
&\leq \frac{C\|\phi\|_*^2}{\lambda^{2\tau}}\int_{D_1}t^{1-2s}\big(\sum\limits_{i=1}^k\frac{1}{(1+|y-x_i|)^{\frac{N-2s}{2}+\tau}}\big)^2\\
&\leq \frac{C\|\phi\|_*^2}{\lambda^{2\tau}}\int_{D_1}\frac{t^{1-2s}k^2}{(1+|y-x_1|)^{N-2s+2\tau}}\\
&\leq \frac{Ck\|\phi\|_*^2}{\lambda^{\tau}}.
\end{aligned}
\end{equation*}
By using the mean value theorem of integrals, there is a $\rho=\rho(\delta)\in(2\delta, 5\delta)$ such that
\begin{equation*}
\begin{aligned}
\int_{\partial''\mathcal B^{+}_{\rho}}t^{1-2s}|\nabla\widetilde{\phi}|^2dS\leq \frac{Ck\|\phi\|_*^2}{\lambda^{\tau}}.
\end{aligned}
\end{equation*}
\end{proof}

\section{Energy expansion}

In this section, we give some estimates of the energy expansions for
$$\langle I'(Z_{\overline{r},\overline{y}'',\lambda}+\phi(\overline{r},\overline{y}'',\lambda)),\frac{\partial Z_{\overline{r},\overline{y}'',\lambda}}{\partial\lambda}\rangle,$$ $$\langle I'(Z_{\overline{r},\overline{y}'',\lambda}+\phi(\overline{r},\overline{y}'',\lambda)),\frac{\partial Z_{\overline{r},\overline{y}'',\lambda}}{\partial\overline{r}}\rangle \hbox{and} \langle I'(Z_{\overline{r},\overline{y}'',\lambda}+\phi(\overline{r},\overline{y}'',\lambda)),\frac{\partial Z_{\overline{r},\overline{y}'',\lambda}}{\partial\overline{y}''}\rangle.$$

\begin{lemma}\label{exp1}
If $N>4s+\tau$, then
\[
\frac{\partial I(Z_{\bar{r},\bar{y}'',\lambda})}{\partial \lambda}=k\left(-\frac{B_1}{\lambda^{2s+1}}V(\bar{r},\bar{y}'')+ \sum_{j=2}^{k}\frac{B_2}{\lambda^{N-2s+1}|x_j-x_1|^{N-2s}}+O\bigg(\frac{1}{\lambda^{2s+1+\sigma}}\bigg)\right),
\]
where $B_1$ and $B_2$ are some positive constants.
\end{lemma}
\begin{proof}
By a direct computation, we have
\begin{equation}
\begin{split}
&\quad \frac{\partial I(Z_{\bar{r},\bar{y}'',\lambda})}{\partial \lambda}\\&=\frac{\partial I(Z^*_{\bar{r},\bar{y}'',\lambda})}{\partial \lambda}+O(\frac{k}{\lambda^{2s+1+\sigma}})\\&
=\ds \int_{\mathbb{R}^N}V(y)Z^*_{\bar{r},\bar{y}'',\lambda}\frac{\partial Z^*_{\bar{r},\bar{y}'',\lambda}}{\partial \lambda}-\ds \int_{\mathbb{R}^N}\big((Z^*_{\bar{r},\bar{y}'',\lambda})^{2^*_s-1}-\sum_{j=1}^kU_{x_j,\lambda}^{2^*_s-1}\big)
\frac{\partial Z^*_{\bar{r},\bar{y}'',\lambda}}{\partial \lambda}+O(\frac{k}{\lambda^{2s+1+\sigma}})\\
&=I_1-I_2+O(\frac{k}{\lambda^{2s+1+\sigma}}).
\end{split}
\end{equation}

For the term $I_1$, By Lemma \ref{Lemma appendix1}, we can check that
\begin{equation}\label{VV}
\begin{array}{ll}
\ds I_1
&=k\bigg(\ds \int_{\mathbb{R}^N}V(y)U_{x_1,\lambda}\frac{\partial U_{x_1,\lambda}}{\partial \lambda}+O\big(\frac{1}{\lambda}\ds \int_{\mathbb{R}^N}U_{x_1,\lambda}\ds \sum_{j=2}^kU_{x_j,\lambda}\big)\bigg)\\
&=k\bigg(\ds \frac{V(\bar{r},\bar{y}'')}{\lambda^{2s+1}}\ds \frac{(N-2s)C^2(N,s)} {2}\int_{\mathbb{R}^N}\frac{\lambda^N(1-\lambda^2|y-x_1|^2)}{(1+\lambda^2|y-x_1|^2)^{N-2s+1}}dy\\
&\quad\quad\quad+O\big(\frac{1}{\lambda^{2s+1}}\ds \sum_{j=2}^k\frac{1}{(\lambda|x_1-x_j|)^{N-4s-\sigma}}\big)
+O(\frac{k}{\lambda^{2s+1+\sigma}})\bigg)\\
&=k\bigg(\ds-\frac{B_1V(\bar{r},\bar{y}'')}{\lambda^{2s+1}}\ds+O(\frac{k}{\lambda^{2s+1+\sigma}})\bigg),\\
\end{array}
\end{equation}
where $B_1=-\frac{(N-2s)C^2(N,s)} {2}\displaystyle\int_{\mathbb{R}^N}\frac{\lambda^N(1-\lambda^2|y-x_1|^2)}{(1+\lambda^2|y-x_1|^2)^{N-2s+1}}dy>0$.

Next, we estimate $I_2$.
$$\aligned
I_2&=k\ds \int_{\Omega_1}\big((Z^*_{\bar{r},\bar{y}'',\lambda})^{2^*_s-1}-\ds\sum_{j=1}^kU_{x_j,\lambda}^{2^*_s-1}\big)
\frac{\partial Z^*_{\bar{r},\bar{y}'',\lambda}}{\partial \lambda}\\&=k\left(\ds \int_{\Omega_1}(2^*_s-1)U_{x_1,\lambda}^{2^*_s-2}\ds\sum_{j=2}^kU_{x_j,\lambda}\frac{\partial U_{x_1,\lambda}}{\partial \lambda}+O\bigg(\frac{1}{\lambda^{2s+1+\sigma}}\bigg)\right)\\&=k\left(-\ds \sum_{j=2}^k\frac{B_2}{\lambda^{N-2s+1}|x_j-x_1|^{N-2s}}+O\bigg(\frac{1}{\lambda^{2s+1+\sigma}}\bigg)\right),
\endaligned$$
for some constant $B_2>0$.

Thus, we obtain that
\[
\frac{\partial I(Z_{\bar{r},\bar{y}'',\lambda})}{\partial \lambda}=k\left(-\frac{B_1}{\lambda^{2s+1}}V(\bar{r},\bar{y}'')+\ds \sum_{j=2}^{k}\frac{B_2}{\lambda^{N-2s+1}|x_j-x_1|^{N-2s}}+O\bigg(\frac{1}{\lambda^{2s+1+\sigma}}\bigg)\right).
\]

\end{proof}

\begin{lemma} \label{exp2}  We have
\begin{equation} \label{energy expansion1}
\begin{aligned}
&\langle I'(Z_{\overline{r},\overline{y}'',\lambda}+\phi),\frac{\partial Z_{\overline{r},\overline{y}'',\lambda}}{\partial\lambda}\rangle\\
=&\langle I'(Z_{\overline{r},\overline{y}'',\lambda}),\frac{\partial Z_{\overline{r},\overline{y}'',\lambda}}{\partial\lambda}\rangle+O(\frac{k}{\lambda^{2s+1+\sigma}})\\
=&k\left(-\frac{B_1}{\la^{2s+1}}V(\bar{r},\bar{y}'')+ \sum_{j=2}^{k}\frac{B_2}{\la^{N-2s+1}|x_j-x_1|^{N-2s}}+O\big(\frac{1}{\la^{2s+1+\sigma}}\big)\right)
\\=&k\left(-\frac{B_1}{\la^{2s+1}}V(\bar{r},\bar{y}'')+ \frac{B_3k^{N-2s}}{\la^{N-2s+1}}+O\big(\frac{1}{\la^{2s+1+\sigma}}\big)\right),
\end{aligned}
\end{equation}
where $B_1$ and $B_2$ are the same constants in Lemma \ref{exp1}, $B_3>0.$
\end{lemma}
\begin{proof}
By symmetry, we have
$$\aligned
&\quad \langle I'(Z_{\overline{r},\overline{y}'',\lambda}+\phi),\frac{\partial Z_{\overline{r},\overline{y}'',\lambda}}{\partial\lambda}\rangle\\
&=\int_{\mathbb{R}^{N}}\big((-\Delta)^s u_k+V(r,y'')u_k-(u_k)_{+}^{2^*_s-1}\big)\frac{\partial Z_{\bar{r},\bar{y}'',\la}}{\partial \la}\\&=\bigg\langle I'(Z_{\bar{r},\bar{y}'',\la}),\frac{\partial Z_{\bar{r},\bar{y}'',\la}}{\partial \la}\bigg\rangle+k\bigg\langle(-\Delta)^s \phi+V(r,y'')\phi-(2^*_s-1)Z_{\bar{r},\bar{y}'',\la}^{2^*_s-2}\phi,\frac{\partial Z_{x_1,\la}}{\partial \la}\bigg\rangle\\&\quad-\ds \int_{\mathbb{R}^{N}}\bigg((Z_{\bar{r},\bar{y}'',\la}+\phi)_+^{2^*_s-1}-Z_{\bar{r},
\bar{y}'',\la}^{2^*_s-1}-(2^*_s-1)Z_{\bar{r},\bar{y}'',\la}^{2^*_s-2}\phi\bigg)\frac{\partial Z_{\bar{r},\bar{y}'',\la}}{\partial \la}\\&:=\bigg\langle I'(Z_{\bar{r},\bar{y}'',\la}),\frac{\partial Z_{\bar{r},\bar{y}'',\la}}{\partial \la}\bigg\rangle+kJ_1-J_2.
\endaligned$$

By \eqref{aum16} and \eqref{equality12}, we have
\[
J_1=O(\frac{||\phi||_*}{\la^{1+s+\sigma}})=O(\frac{1}{\la^{2s+1+\sigma}}).
\]

Note that $(1+t)_+^\gamma-1-\gamma t=O(t^2)$ for all $t\in \mathbb{R}^N$ if $1<\gamma\leq2$, and $|(1+t)_+^\gamma-1-\gamma t|\leq C(t^2+|t|^\gamma)$ for all $t\in \mathbb{R}^N$ if $\gamma>2$. So,
if $2^*_s\leq3$,  we have
$$\aligned
|J_2|&= \bigg|\ds \int_{\mathbb{R}^{N}}\bigg((Z_{\bar{r},\bar{y}'',\la}+\phi)_+^{2^*_s-1}-Z_{\bar{r},
\bar{y}'',\la}^{2^*_s-1}-(2^*_s-1)Z_{\bar{r},\bar{y}'',\la}^{2^*_s-2}\phi\bigg)\frac{\partial Z_{\bar{r},\bar{y}'',\la}}{\partial \la}\bigg| \\
&\leq C\ds \int_{\mathbb{R}^{N}}Z_{\bar{r},
\bar{y}'',\la}^{2^*_s-3}\phi^2\big|\frac{\partial Z_{\bar{r},\bar{y}'',\la}}{\partial \la}\big|\\
& \leq \frac{C||\phi||_*^2}{\la}\ds \int_{\mathbb{R}^{N}}\left(\ds \sum_{j=1}^k\frac{\la^{\frac{N-2s}{2}}}{(1+\la|y-x_j|)^{N-2s}}\right)^{2^*_s-2}\left(\ds \sum_{i=1}^k\frac{\la^{\frac{N-2s}{2}}}{(1+\la|y-x_i|)^{\frac{N-2s}{2}+\tau}}\right)^2
\\&\leq \frac{C||\phi||_*^2}{\la}\ds \int_{\mathbb{R}^{N}}\la^N \sum_{j=1}^k\frac{1}{(1+\la|y-x_j|)^{4s}}\sum_{i=1}^k\frac{1}{(1+\la|y-x_i|)^{N-2s+\tau}}\\
&\leq \frac{Ck||\phi||_*^2}{\la}=O\left(\frac{k}{\la^{2s+1+\sigma}}\right).
\endaligned$$
If $2^*_s>3$, we have
$$\aligned
|J_2|&\leq C\ds \int_{\mathbb{R}^{N}}\left(Z_{\bar{r},
\bar{y}'',\la}^{2^*_s-3}\phi^2\big|\frac{\partial Z_{\bar{r},\bar{y}'',\la}}{\partial \la}\big|+|\phi|^{2^*_s-1}\big|\frac{\partial Z_{\bar{r},\bar{y}'',\la}}{\partial \la}\big|\right)\\&=O\left(\frac{k}{\la^{2s+1+\sigma}}\right).
\endaligned$$
Thus we obtain
\[
\bigg\langle I'(Z_{\bar{r},\bar{y}'',\la}+\phi),\frac{\partial Z_{\bar{r},\bar{y}'',\la}}{\partial \la}\bigg\rangle=\bigg\langle I'(Z_{\bar{r},\bar{y}'',\la}),\frac{\partial Z_{\bar{r},\bar{y}'',\la}}{\partial \la}\bigg\rangle+O\left(\frac{k}{\la^{2s+1+\sigma}}\right).
\]
Combining Lemma \ref{exp1}, we finish the proof.
\end{proof}

Note that $Z_{i,l}=O(\lambda Z_{x_i,\lambda}), \ l=2,\ldots,N.$ Similarly, we can prove the following lemma:
\begin{lemma} \label{exp3}
We have
\begin{equation} \label{energy expansion2}
\begin{aligned}
&\quad\langle I'(Z_{\overline{r},\overline{y}'',\lambda}+\phi),\frac{\partial Z_{\overline{r},\overline{y}'',\lambda}}{\partial\overline{r}}\rangle\\
&=\langle I'(Z_{\overline{r},\overline{y}'',\lambda}),\frac{\partial Z_{\overline{r},\overline{y}'',\lambda}}{\partial\overline{r}}\rangle+O(\frac{k}{\lambda^{s+\sigma}})\\
&=k\big(\frac{B_1}{\lambda^{2s}}\frac{\partial V(\overline{r},\overline{y}'')}{\partial\overline{r}}+\sum\limits_{j=2}^k\frac{B_2}{\overline{r}\lambda^{N-2s}|x_1-x_j|^{N-2s}}+O(\frac{1}{\lambda^{s+\sigma}})\big),
\end{aligned}
\end{equation}
and
\begin{equation} \label{energy expansion3}
\begin{aligned}
\langle I'(Z_{\overline{r},\overline{y}'',\lambda}+\phi),\frac{\partial Z_{\overline{r},\overline{y}'',\lambda}}{\partial\overline{y}''_j}\rangle
=\langle I'(Z_{\overline{r},\overline{y}'',\lambda}),\frac{\partial Z_{\overline{r},\overline{y}'',\lambda}}{\partial\overline{y}''_j}\rangle+O(\frac{k}{\lambda^{s+\sigma}})=k\big(\frac{B_1}{\lambda^{2s}}\frac{\partial V(\overline{r},\overline{y}'')}{\partial\overline{y}''_j}+O(\frac{1}{\lambda^{s+\sigma}})\big),
\end{aligned}
\end{equation}
where $B_1$ and $B_2$ are the same constants in Lemma \ref{exp1}.
\end{lemma}
\end{appendices}

\bibliographystyle{springer}
\bibliography{mrabbrev,literatur}
\newcommand{\noopsort}[1]{} \newcommand{\printfirst}[2]{#1}
\newcommand{\singleletter}[1]{#1} \newcommand{\switchargs}[2]{#2#1}

\end{document}